\documentclass[a4paper,10pt,reqno]{amsart}

\usepackage[UKenglish]{babel}
\usepackage{amsmath}
\usepackage{amssymb}
\usepackage{amsthm}
\usepackage{verbatim}
\usepackage{stackrel}
\usepackage[arrow, matrix, curve]{xy} 

\usepackage[shortlabels]{enumitem}
\usepackage[nospace,noadjust]{cite}
\usepackage{comment}	
\usepackage{hyperref}
\hypersetup{
    colorlinks=true,
    linkcolor=blue,
    citecolor=red,
    filecolor=magenta,      
    urlcolor=cyan,
    pdftitle={Eventual positivity and spectrum},
    bookmarks=true,
    linktocpage=true,
}

\usepackage[utf8]{inputenc}

\usepackage{amssymb}
\usepackage{amsmath}
\usepackage{amsthm}
\usepackage{bbm}
\usepackage{verbatim, stackrel}

\usepackage[shortlabels]{enumitem}
\usepackage{marginnote}
\usepackage{color}

\usepackage{stmaryrd}
\usepackage{tikz}
\usepackage{tikz-cd}
\usepackage{tkz-graph}

\newcommand{\bbC}{\mathbb{C}}

\newcommand{\bbN}{\mathbb{N}}

\newcommand{\bbR}{\mathbb{R}}

\newcommand{\bbZ}{\mathbb{Z}}

\newcommand{\calL}{\mathcal{L}}

\newcommand{\calU}{\mathcal{U}}

\DeclareMathOperator{\one}{{\mathbbm{1}}} 
\DeclareMathOperator{\re}{Re} 
\DeclareMathOperator{\dist}{dist} 
\newcommand{\argument}{\mathord{\,\cdot\,}} 
\newcommand{\dx}{\;\mathrm{d}} 
\DeclareMathOperator{\sgn}{sgn} 
\DeclareMathOperator{\Fix}{Fix} 
\newcommand{\norm}[1]{\left\lVert #1 \right\rVert} 
\newcommand{\modulus}[1]{\left\lvert #1 \right\rvert} 
\newcommand{\duality}[2]{\left\langle#1\, ,\, #2\right\rangle} 
\newcommand{\dom}[1]{\operatorname{dom}\left(#1\right)} 
\DeclareMathOperator{\Ima}{Rg} 
\newcommand\restrict[1]{\raisebox{-.5ex}{$|$}_{#1}} 

\newcommand{\spec}{\sigma} 

\newcommand{\Res}{\mathcal{R}} 
\newcommand{\pntSpec}{\spec_{\operatorname{pnt}}} 
\newcommand{\essSpec}{\spec_{\operatorname{ess}}} 
\newcommand{\perSpec}{\spec_{\operatorname{per}}} 
\newcommand{\perpntSpec}{\spec_{\operatorname{per,pnt}}} 
\newcommand{\spr}{r} 
\newcommand{\spb}{s} 
\newcommand{\essSpb}{\spb_{\operatorname{ess}}} 
\newcommand{\gbd}{\omega_0} 

\theoremstyle{definition}
\newtheorem{definition}{Definition}[section]
\newtheorem{remark}[definition]{Remark}
\newtheorem{remarks}[definition]{Remarks}
\newtheorem*{remark*}{Remark}
\newtheorem*{remarks*}{Remarks}
\newtheorem{example}[definition]{Example}

\theoremstyle{plain}
\newtheorem{proposition}[definition]{Proposition}
\newtheorem{lemma}[definition]{Lemma}
\newtheorem{theorem}[definition]{Theorem}
\newtheorem{corollary}[definition]{Corollary}

\numberwithin{equation}{section} 

\begin{document}

\title[Eventual positivity and spectrum]{Eventually positive semigroups: spectral and asymptotic analysis}
\author{Sahiba Arora}
\address{Sahiba Arora, Department of Applied Mathematics, University of Twente, 217, 7500 AE, Enschede, The Netherlands}
\email{sahiba.arora@math.uni-hannover.de}
\date{\today}
\begin{abstract}
    The spectral theory of semigroup generators is a crucial tool for analysing the asymptotic properties of operator semigroups. Typically, Tau\-berian theorems, such as the ABLV theorem, demand extensive information about the spectrum to derive convergence results. However, the scenario is significantly simplified for positive semigroups on Banach lattices. This observation extends to the broader class of eventually positive semigroups -- a phenomenon observed in various concrete differential equations. In this paper, we investigate the spectral and asymptotic properties of eventually positive semigroups, focusing particularly on the persistently irreducible case. Our findings expand upon the existing theory of eventual positivity, offering new insights into the cyclicity of the peripheral spectrum and asymptotic trends. Notably, several arguments for positive operators and semigroups do not apply in our context, necessitating the use of ultrapower arguments to circumvent these challenges.
\end{abstract}

\maketitle

\section{Introduction}

Owing to a variety of applications -- from physical and biological models to probability theory and transport phenomena -- positivity is an omnipresent topic in the theory of $C_0$-semigroups. While the study of positive operator semigroups on Banach lattices is classical and can be found, for instance, in the monographs \cite{Nagel1986, BatkaiKramarRhandi2017}, the study of eventually positive semigroups (in the infinite-dimensional setting) has emerged only in the last decade. This phenomenon was first observed by Daners in 2014 \cite{Daners2014} for a Dirichlet-to-Neumann semigroup. Since then, a systematic theory of eventually positive semigroups has rapidly emerged starting from the works of Daners, Glück, and Kennedy \cite{DanersGlueckKennedy2016a, DanersGlueckKennedy2016b}; see \cite{Glueck2022} for a survey. The theory of eventual positivity has been complemented by a wide range of examples, such as second-order operators with (non-local) Robin boundary conditions \cite[Theorem~1.1]{GlueckMui2024},
Laplacians with non-local boundary conditions and Dirichlet bi-Laplacians \cite[Section~4]{DanersGlueck2018b}, semigroups on metric graphs (\cite[Proposition~4.5]{BeckerGregorioMugnolo2021}, \cite[Section~6]{GregorioMugnolo2020a}, and \cite[Proposition~3.7]{GregorioMugnolo2020b}), various delay differential equations (\cite[Section~4]{DanersGlueck2018b} and  \cite[Section~11.6]{Glueck2016}), Laplacians with point-interactions \cite[Proposition~2]{HusseinMugnolo2020},
and elliptic operators with Wentzell boundary conditions \cite[Section~7]{DenkKunzePloss2021}, \cite[Section~4.3]{Ploss2024}, and \cite[Section~7]{KunzeMuiPloss2025}.

The notion of eventual positivity has been expanded in various directions, such as asymptotic positivity \cite[Sections~8-9]{DanersGlueckKennedy2016b}, local eventual positivity \cite{Arora2022, Mui2023, DanersGlueckMui2023}, eventual domination \cite{GlueckMugnolo2021, AroraGlueck2023b}, and eventual positivity of Ces{\`a}ro means \cite[Section~5]{AroraGlueck2023b}. Moreover, recently the notion of persistent irreducibility was introduced in \cite{AroraGlueck2024} as a version of irreducibility more suited in the context of eventual positivity. 

\subsection*{Contributions and organisation of the article}

The renowned works of Perron and Frobenius on positive matrices, along with those of Kre\u{\i}n and Rutman's on positive operators, demonstrate that positivity properties have a significant impact on the spectrum. Indeed, this feature is enjoyed by positive semigroups as well. Motivated by these insights, the spectral analysis of eventually positive semigroups was undertaken in \cite[Section~7]{DanersGlueckKennedy2016a} and \cite{AroraGlueck2021a} in order to extend known spectral and asymptotic results for positive semigroups. Likewise, spectral properties of persistently irreducible eventually positive semigroups were examined in \cite[Section~4]{AroraGlueck2024}. The challenge to prove these results is in the fact that only a limited number of techniques used in the positive case carry over to the eventual positivity case.

One such roadblock lies in not knowing whether the peripheral spectrum of an eventually positive semigroup is cyclic under the same general conditions as in the positive case. Another such barrier lies in the fact that the generator of a positive semigroup is a resolvent positive operator, a property that greatly facilitates various proofs of spectral results for positive semigroups. The first problem was circumvented in \cite{AroraGlueck2021a} by proving a Niiro-Sawashima type theorem for operators whose powers eventually become positive. In this paper, we tackle the second issue by employing ultrapower techniques to generalise results previously known only for the positive case. Our contributions can be highlighted as follows:

\begin{enumerate}[\upshape (a)]
    \item To warm up, we give a reminder of the notion of (persistent) irreducibility and then limit ourselves to matrix semigroups in Section~\ref{sec:finite-dimensions}. More precisely, we give an example of a matrix $A$ on $E=\bbR^4$ such that the corresponding matrix semigroup is irreducible, yet the condition
    \[
        \forall\ 0\lneq f \in E, 0\lneq \varphi\in E', \exists\ t\ge 0: \duality{\varphi}{e^{tA}f}\ne 0
    \]
    is not fulfilled. The significance of this example lies in the fact that the corresponding  semigroup lacks (eventual) positivity. This is important because, along with analyticity, (eventual) positivity is sufficient for the above condition to follow from irreducibility \cite[Proposition~3.11]{AroraGlueck2024}.

    \item Next, we show that one can always associate a positive pseudo-resolvent to a uniformly eventually positive bounded semigroup that retains useful information about the spectrum. This is the content of Section~\ref{sec:ultrapowers}.

    \item The ultrapower techniques developed in Section~\ref{sec:ultrapowers} are then employed in Sections~\ref{sec:peripheral-spectrum} and~\ref{sec:domination}. Not only do we show, in Theorem~\ref{thm:cyclicity-peripheral}, that for a uniformly eventually positive semigroup, boundedness implies cyclicity of the peripheral spectrum (provided spectral bound of the generator is $0$) but also generalise a spectral result of Räbiger-Wolff about domination of positive semigroups to eventually positive semigroups that are merely satisfy eventually domination (Theorem~\ref{thm:spectrum-under-domination}).
    
    \item Regarding cyclicity, we are also  able to show that for uniformly eventually positive semigroups that are persistently irreducible, if the spectral bound of the generator is a pole and the semigroup is bounded, then the peripheral spectrum is cyclic (Proposition~\ref{prop:peripheral-in-point} and Theorem~\ref{thm:cyclic-peripheral-point}).

    \item In Section~\ref{sec:convergence}, we prove a new result about strong convergence of eventually positive semigroups and also study the asymptotic behaviour of persistently irreducible semigroups, an endeavour that was not taken up in \cite{AroraGlueck2024}. Finally, in Section~\ref{sec:non-empty-spectrum}, we give a sufficient condition to ensure that the spectrum of a persistently irreducible semigroup is non-empty.
\end{enumerate}

The rest of this section is devoted to recalling the concept of eventual positivity as well as introducing certain terminology.

\subsection*{Notation and preliminaries}

We assume the reader is familiar with the theories of Banach lattices, operator semigroups, and spectral theory of closed operators. Our notations and terminology are standard and can be found in the monographs \cite{Meyer-Nieberg1991, Schaefer1974}, \cite{EngelNagel2000}, and \cite[Chapter~VIII]{Yosida1980} respectively. 

Let $E$ be a Banach lattice with positive cone $E_+$. We use the notation $f\gneq 0$ to mean that $f\ge 0$ but $f\ne 0$ or equivalently, $f\in E_+\setminus \{0\}$. A subspace $I$ of $E$ is called a \emph{(lattice) ideal} if for each $x,y\in E$, the inequality $0\le y\le x$ and $x\in I$ implies $y\in I$.
The \emph{principal ideal} generated by an element $u\in E_+$ defined as
\[
    E_u := \{ f\in E: \text{ there exists }c>0\text{ such that }\modulus{f}\le cu\}
\]
is a Banach lattice when equipped with the \emph{gauge norm}
\[
    \norm{f}_u := \inf \{c>0 : \modulus{f}\le cu\}, \qquad (f\in E_u),
\]
that embeds continuously into $E$. If $E_u$ is also dense in $E$, then we call $u$ a \emph{quasi-interior} point of $E_+$. This is equivalent by \cite[Theorem~II.6.3]{Schaefer1974} to $\duality{\varphi}{u}>0$ for each $\varphi \in E'_+$; a fact we use repeatedly without quoting.
Let $E$ and $F$ be complex Banach lattices with real parts $E_{\bbR}$ and $F_{\bbR}$ respectively. A linear operator $A$ between $E$ and $F$ with domain $\dom A$ is called \emph{real} if $x+iy\in \dom A$ with $x,y \in E_{\bbR}$ implies $x,y\in \dom A$ and $A$ maps $\dom A \cap E_{\bbR}$ into $F_{\bbR}$. A natural example of a real operator is a {positive operator}, i.e., one that maps the positive cone of the domain to the positive cone of the co-domain. A positive functional $\varphi \in E'$ is said to be \emph{strictly positive} if its kernel contains no positive non-zero elements.

Let $(e^{tA})_{t\ge 0}$ be a semigroup on a complex Banach lattice $E$. We say that $(e^{tA})_{t\ge 0}$ is \emph{real} if $e^{tA}$ is a real operator for each $t\ge 0$. Furthermore, $(e^{tA})_{t\ge 0}$ is called \emph{individually eventually positive} if for each $0\le f\in E$, there exists $t_0\ge 0$ such that $e^{tA}f\ge 0$ for all $t\ge t_0$; if $t_0$ can be chosen independent of $f$, then the semigroup is said to be \emph{uniformly eventually positive}. In finite-dimensional spaces, individual and uniform eventual positivity are obviously equivalent and they were first studied in \cite{NoutsosTsatsomeros2008}. In contrast, there exist semigroups that are individually but not uniformly eventually positive in the infinite-dimensional setting \cite[Examples~5.7 and~5.8]{DanersGlueckKennedy2016a}. The investigation of eventually positive semigroups on complex Banach lattices was initiated in a series of papers \cite{DanersGlueckKennedy2016a, DanersGlueckKennedy2016b, DanersGlueck2017, DanersGlueck2018a, DanersGlueck2018b}, and since then the theory has not only developed but also branched out \cite{Arora2022, Mui2023, DanersGlueckMui2023, GlueckMui2024, AroraGlueck2021a, AroraGlueck2023b, AroraGlueck2024, AroraGlueck2022b, GlueckMugnolo2021, ArnoldCoine2023, Peruzzetto2022, PappuRastogiSrivastava2025}.

\section{Persistent irreducibility without (eventual) positivity}
    \label{sec:finite-dimensions}

Let $(e^{tA})_{t\ge 0}$ be a $C_0$-semigroup on a Banach lattice $E$. Recall that $(e^{tA})_{t\ge 0}$ is called \emph{irreducible} if it leaves no non-empty proper closed ideal of $E$ invariant. Irreducibility  is usually studied for the class of positive semigroups, see \cite[Section~C-III-3]{Nagel1986} or \cite[Section~14.3]{BatkaiKramarRhandi2017}. Recently, though, irreducibility has been investigated for the class of eventually positive semigroups in \cite{AroraGlueck2024}. In this context, the stronger notion of \emph{persistently irreducible} semigroups turns out to be more relevant:~$(e^{tA})_{t\ge 0}$ is said to be \emph{persistently irreducible} if it leaves no non-empty proper closed ideal of $E$ \emph{eventually invariant}, i.e., if $I$ is a closed ideal in $E$ and there exists $t_0\ge 0$ such that $e^{tA}I\subseteq I$ for all $t\ge t_0$, then $I=\varnothing$ or $I=E$. An example of an irreducible semigroup that is not persistently irreducible can be found in \cite[Example~3.10]{AroraGlueck2024} and an example of a non-positive persistently irreducible semigroup is constructed in \cite[Example~5.6]{AroraGlueck2024}. Furthermore, the semigroups considered in \cite[Section~6]{DanersGlueckKennedy2016b} and \cite[Section~4]{DanersGlueck2018b} are all persistently irreducible due to \cite[Remark~3.4]{AroraGlueck2024}.

In this short section, we restrict our attention to the finite-dimensional case. Here, it can be easily checked whether a matrix or the corresponding matrix semigroup is irreducible by looking at the associated directed graph. In fact, in this case the notions of irreduciblity and persistent irreducibility turn out to be equivalent (Proposition~\ref{prop:irreducibility-matrix-semigroup}) without any kind of positivity assumptions.  Thereafter, we give an example to show that the eventual positivity assumption is needed for the equivalences in \cite[Example~3.11]{AroraGlueck2024} to hold.

Let $A=(a_{ij})$ be a real matrix of dimension $d$. The \emph{weighted directed graph associated to $A$} is a graph with $d$ vertices $1,2,\ldots,d$ such that there is an edge $e_{ij}$ with weight $a_{ij}$ from vertex $i$ to vertex $j$ if and only if $a_{ij}\ne 0$. Conversely, every weighted directed graph corresponds to a square matrix. A directed graph is called \emph{strongly connected} if every vertex can be reached from any other vertex, i.e., if $u,v$ are distinct vertices, then there is a path in each direction between $u$ and $v$.

\begin{proposition}
    \label{prop:irreducibility-matrix-semigroup}
    The following conditions are equivalent for a matrix $A \in \bbR^{d\times d}$.
    \begin{enumerate}[\upshape (i)]
        \item The matrix $A$ is irreducible.
        \item There is no permutation matrix $P$ and square matrices $X$ and $Z$ such that
            \[
                PAP^{-1} = 
                \begin{bmatrix}
                    X & Y\\
                    0 & Z
                \end{bmatrix}.
            \]
        \item The directed graph associated to $A$ is strongly connected.
        \item The matrix $\modulus A$ is irreducible.
        \item The matrix semigroup $(e^{tA})_{t \ge 0}$ is irreducible.
        \item The matrix semigroup $(e^{tA})_{t \ge 0}$ is persistently irreducible.
    \end{enumerate}
\end{proposition}

\begin{proof}
    The equivalences (i)-(iii) are well-known and can be found, for instance, in \cite[Lemma~5.10]{BatkaiKramarRhandi2017} and \cite[Theorem~1.2.3 and Page~172]{BrualdiCvetkovic2009}. Obviously, the directed graph associated with $A$ is strongly connected if and only if the directed graph associated with $\modulus A$ is strongly connected, so we even get that (i)-(iv) are equivalent. 
    In addition, since every matrix semigroup is analytic, so by the identity theorem of analytic functions, conditions (v) and (vi) are equivalent. Finally, due to the exponential formula, a closed ideal is $A$-invariant if and only if it is $(e^{tA})_{t\ge0}$-invariant, hence (i) is also equivalent to (v).
\end{proof}

For a $C_0$-semigroup $(e^{tA})_{t\ge 0}$ on a Banach lattice $E$, the condition
\begin{equation}
    \label{eq:weak-condition-arbitrary-times}
    \text{for each }0\lneq f \in E, 0\lneq \varphi\in E', \text{ there exists }t\ge 0: \duality{\varphi}{e^{tA}f}\ne 0
\end{equation}
implies that $(e^{tA})_{t\ge 0}$ is irreducible \cite[Proposition~3.3]{AroraGlueck2024}. The converse is true if $(e^{tA})_{t\ge 0}$ is individually eventually positive and analytic \cite[Proposition~3.11]{AroraGlueck2024}. The following example shows that one cannot drop the eventual positivity assumption in  \cite[Proposition~3.11]{AroraGlueck2024}, even in finite-dimensions.

\begin{example}
    Consider the symmetric matrix
    \[
        A = 
            \begin{bmatrix}
                0 & 0 & -1 & 1\\
                0 & 0 & 1 & -1\\
                -1 & 1 & 0 & 0\\
                1 & -1 & 0 & 0
            \end{bmatrix}.
    \]
    It can be checked that the largest eigenvalue $r=2$
    does not have a positive eigenvector and so $(e^{tA})_{t \ge 0}$ is not individually eventually positive \cite[Theorem~7.7]{DanersGlueckKennedy2016a}. Moreover, the kernel of $A=A^T$ contains the positive non-zero vectors
    \[
        f : = 
            \begin{bmatrix}
                1 \\
                1 \\
                0 \\
                0
            \end{bmatrix}
            \qquad   \text{and}  \qquad
        \varphi :=
            \begin{bmatrix}
                0 \\
                0 \\
                1 \\
                1
            \end{bmatrix}
    \]
    and hence, $\duality{\varphi}{e^{tA}f}=\duality{\varphi}{f}=0$ for all $t\ge 0$. In other words,~\eqref{eq:weak-condition-arbitrary-times} fails.

    On the other hand, the directed graph associated with $A$ is given by
    
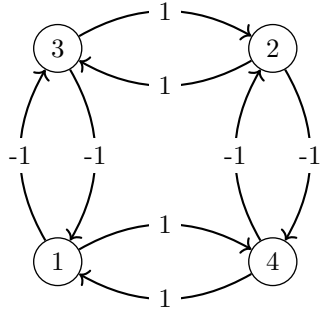
\begin{center}
        \begin{tikzpicture}
            \GraphInit[vstyle=Dijkstra]
            \SetGraphUnit{2}
            \begin{scope}[rotate=45]
                \Vertices{circle}{2,3,1,4}
            \end{scope}
            \Edge[style={->,bend left},label=-1](1)(3)
            \Edge[style={->,bend left},label=1](1)(4)
            \Edge[style={->,bend left},label=1](2)(3)
            \Edge[style={->,bend left},label=-1](2)(4)
            \Edge[style={->,bend left},label=-1](3)(1)
            \Edge[style={->,bend left},label=1](3)(2)
            \Edge[style={->,bend left},label=1](4)(1)
            \Edge[style={->,bend left},label=-1](4)(2)
        \end{tikzpicture}
    \end{center}
    which is strongly connected, hence $(e^{tA})_{t \ge 0}$ is irreducible by Proposition~\ref{prop:irreducibility-matrix-semigroup}.
\end{example}

\section{Asymptotic positivity and ultrapowers}
    \label{sec:ultrapowers}

The one-to-one correspondence between the positivity of a semigroup and the positivity of the associated resolvent facilitates the proof of various spectral results. This relationship becomes problematic when dealing with eventually positive semigroups as the associated resolvent operators need not even be positive in some neighbourhood of a spectral value \cite[Example~8.2]{DanersGlueckKennedy2016a}.
Nevertheless, to every uniformly eventually positive -- even uniformly asymptotically positive -- semigroup, one can associate a positive pseudo-resolvent on a suitable Banach lattice that retains desirable properties. This can be achieved by employing ultrapowers that have proven to be a valuable tool in the spectral analysis of eventual positivity \cite{AroraGlueck2021a, Glueck2017}. We start this section by fixing our notation and terminology about pseudo-resolvents and ultrapowers and recalling simple, yet important properties that we require in the sequel. 

Let $E$ be a Banach space and $D\subseteq \bbC$ be non-empty. A mapping $R:D\to \calL(E)$ is called a \emph{pseudo-resolvent} if it satisfies the resolvent identity
\[
    R(\lambda)-R(\mu)=(\mu-\lambda)R(\lambda)R(\mu)
\]
for all $\lambda,\mu\in D$. In particular, the resolvent of a linear operator with a non-empty resolvent set is a pseudo-resolvent. A number $\alpha \in \bbC$ is called an \emph{eigenvalue of the pseudo-resolvent} $R$ with corresponding eigenvector $0\ne x\in E$, if
\[
    (\lambda-\alpha)R(\lambda)x=x \quad \text{for all }\lambda \in D.
\]
Actually, this is implied by the seemingly weaker condition
\begin{equation}
    \label{eq:pseudo-resolvent-eigenvalue-sufficient}
    (\lambda-\alpha)R(\lambda)x=x \quad \text{for some }\lambda \in D;
\end{equation}
see \cite[Proposition~C-III.2.6(a)]{Nagel1986}. An important property of pseudo-resolvents that we use throughout is that if $F$ is closed $R$-invariant subspace of $E$ -- denote by $\widehat R(\lambda)$, the induced operators on the quotient space $E/F$ -- then $\left\{\widehat R(\lambda)\right\}_{\lambda \in D}$ and the family of restricted operators $\left\{R(\lambda)\restrict F\right\}_{\lambda\in D}$ both form pseudo-resolvents. If $E$ is a Banach lattice, then $R$ is called a \emph{positive pseudo-resolvent} if there exists a real number $r\in D$ such that $(r,\infty)\subseteq D$ and  $R(\lambda)$ is positive whenever $\lambda>r$. If $\{R(\lambda)\}_{\re \lambda>0}$ is a positive pseudo-resolvent, then
\begin{equation}
    \label{eq:positive-pseudo-resolvent}
        \modulus{ R(\lambda)f }\le R(\re \lambda) \modulus{f} \quad \text{for } \re\lambda>0 \text{ and }f\in E;
\end{equation}
see, for instance, \cite[Proof of Propostion~C-III-2.7]{Nagel1986}. For other properties of (positive) pseudo-resolvents, we refer to \cite[Pages~298-300]{Nagel1986} and \cite[Section~2]{RabigerWolff2000}.

Let $E$ be a Banach space, $\ell^\infty(E)$ denote the Banach space of bounded sequences in $E$ endowed with the supremum norm, and fix a free ultrafilter $\calU$ on $\bbN$. Denoting by $c_{\calU}(E)$, the bounded sequences in $E$ that converge to $0$ along $\calU$, i.e.,
\[
    c_{\calU}(E):= \{ x\in \ell^\infty(E) : \lim_{\calU}x =0\},
\]
the \emph{ultrapower of E with respect to $\calU$} is defined as the quotient space
\[
    E^{\calU} := \ell^\infty(E) / c_{\calU}(E).
\]
For each $x=(x_n)\in \ell^\infty(E)$, the corresponding element in $E^{\calU}$ is denoted by $x^{\calU}= (x_n)^{\calU}$. 
It turns out that $\norm{x^{\calU}}=\lim_{\calU} \norm{x_n}$. 
Also, for each $x\in E$, we write $x^{\calU}:= (x,x,\ldots)^{\calU}$.
Bounded operators $T:E\to F$ between Banach spaces $E$ and $F$ have a natural norm-preserving extension $T^{\calU} \in \calL(E^{\calU}, F^{\calU})$ given by $T^{\calU}x^{\calU}:= (Tx_n)^{\calU}$ for all $x=(x_n)\in \ell^{\infty}(E)$. Moreover, if $E$ and $F$ are Banach lattices, then $T$ is positive if and only if $T^{\calU}$ is positive.
A simple yet important property of ultrapowers that we use throughout is that every pseudo-resolvent $\{R(\lambda)\}_{\lambda \in D}$ on a Banach space gives rise to a pseudo-resolvent $\left\{R(\lambda)^{\calU}\right\}_{\lambda \in D}$ on the corresponding ultrapower \cite[Proposition~2.14(a)]{RabigerWolff2000}. For a comprehensive overview of ultrapowers, we refer the reader to \cite{Heinrich1980}, \cite[Pages~251-253]{Meyer-Nieberg1991}, and \cite[Section~V.1]{Schaefer1974}.

\begin{proposition}
    \label{prop:eigenvector-fixed-space}
    Let $(e^{tA})_{t \ge 0}$ be a bounded $C_0$-semigroup on a complex Banach space $E$.
    Fix a free ultrafilter $\calU$ on $\bbN$.

    If $i\beta \in i\bbR$ is an approximate eigenvalue of $A$ with approximate eigenvector $(x_n)$, then $i\beta$
    is an eigenvalue of the pseudo-resolvent  $\left\{\Res(\lambda, A)^{\calU}\right\}_{\re\lambda>0}$ with eigenvector $(x_n)^{\calU}$. Moreover, there exists a subsequence $(y_{n})$ of $(x_n)$  such that $(y_{n})^{\calU}$ lies in the fixed space of the bounded operator
    \[
        S^\beta:=  \begin{cases}
                    \left(\left(e^{ nA}\right)_n\right)^{\calU} \qquad & \text{if }\beta=0\\
                    \left(\left(e^{\frac{2\pi n}{\beta}A}\right)_n\right)^{\calU} \qquad & \text{if }\beta\ne 0.
                \end{cases}
    \]
\end{proposition}

\begin{remarks}
    \label{rem:eigenvector-fixed-space}
    (a) Since every spectral value on the boundary of the spectrum of a closed operator is its approximate eigenvalue \cite[Propostion~IV.1.10]{EngelNagel2000}, the assertion of Proposition~\ref{prop:eigenvector-fixed-space} can be rephrased as:~If $i\beta \in i\bbR$ is a spectral value of $A$, then it is an eigenvalue of the pseudo-resolvent $\left\{\Res(\lambda, A)^{\calU}\right\}_{\re\lambda>0}$ with an eigenvector in $\Fix S^{\beta}$.
    
    (b) The assertion that $i\beta$
    is an eigenvalue of the extended pseudo-resolvent on $E^{\calU}$ in Proposition~\ref{prop:eigenvector-fixed-space} is actually well-known \cite[Proposition on Page~78]{Nagel1986}. We include the proof for the sake of completeness.
\end{remarks}

For a linear operator $A$ on a Banach space, its \emph{spectral bound} will be denoted by
\[
    \spb(A):= \sup_{\lambda \in \spec(A)} \re \lambda \in [-\infty,\infty].
\]
Keep in mind that if $A$ generates a $C_0$-semigroup with growth bound $\gbd(A)$, then $\spb(A)\le \gbd(A)$.

\begin{proof}[Proof of Proposition~\ref{prop:eigenvector-fixed-space}]
    First of all, as the semigroup is bounded and $i\beta$ is a spectral value of $A$, so $\spb(A)= 0$. 
    Let $(x_n)$ be a normalized approximate eigenvector of $A$ corresponding to $i\beta$. 
    We see at once from the identity
    \[
        x_n - (\lambda-i\beta)\Res(\lambda, A)x_n= \Res(\lambda,A)(i\beta-A)x_n,
    \]
    that $x^{\calU} := (x_n)^{\calU}\in E^{\calU}$ is an eigenvector of the pseudo-resolvent $\left\{\Res(\lambda, A)^{\calU}\right\}_{\re\lambda>0}$ corresponding to the eigenvalue $i\beta$. 
    
    For the second assertion, we simply choose a subsequence $(y_n)$ of $(x_n)$ such that $n\norm{(i \beta - A)y_n}\to 0$ as $n\to \infty$. Now, the identity
    \[
        y_n-e^{-i\beta t}  e^{t A} y_n = \int_0^t e^{-(i\beta-A) s}\left(i\beta-A\right)y_n\dx s
    \]
    holds for all $t\ge 0$ and all $n\in \bbN$ by \cite[Lemma~II.1.9]{EngelNagel2000}. 
    In particular, if $\beta=0$, then
    \[
        \norm{y_n-  e^{n A} y_n}  = \norm{\int_0^n e^{sA}A y_n\dx s}\le   nM\norm{Ay_n} \to 0
    \]
    as $n\to \infty$; where $M:=\sup_{t\ge 0}\norm{e^{tA}}<\infty$ due to boundedness of the semigroup.
    Likewise, if $\beta\ne 0$, we get 
    \begin{align*}
        \norm{y_n-  e^{\frac{2\pi n}{\beta} A} y_n} & = \norm{\int_0^{\frac{2\pi n}{\beta}} e^{-(i\beta-A) s}\left(i\beta-A\right)y_n\dx s}\\
                                                    & \le  \frac{2\pi nM}{\beta}\norm{(i \beta - A)y_n} \to 0
    \end{align*}
    as $n\to \infty$. In either case, $S^{\beta}(y_{n})^{\calU}=(y_{n})^{\calU}$.
\end{proof}

If $\lambda\in \bbC$ is an approximate eigenvalue of a bounded linear operator $T$ on a Banach space $E$, then we know that it is an eigenvalue of $T^{\calU}$ for every free ultrafilter $\calU$ on $\bbN$ by \cite[Theorem~V.1.4(ii)]{Schaefer1974}. In fact, if $\lambda$ is not in the point spectrum of $T$, then it turns out that the eigenspace $\ker(\lambda-T^{\calU})$ is even infinite-dimensional \cite[Lemma~C-III-3.10]{Nagel1986}. We generalise this to the setting of Proposition~\ref{prop:eigenvector-fixed-space}:

\begin{proposition}
    \label{prop:infinite-dimensional-eigenspace-resolvent}
    Let $(e^{tA})_{t \ge 0}$ be a bounded $C_0$-semigroup on a complex Banach space $E$.
    Fix a free ultrafilter $\calU$ on $\bbN$ and on the ultrapower $E^{\calU}$ consider the bounded linear operators 
    \[
        S^\beta:=  \begin{cases}
                    \left(\left(e^{ nA}\right)_n\right)^{\calU} \qquad & \text{if }\beta=0\\
                    \left(\left(e^{\frac{2\pi n}{\beta}A}\right)_n\right)^{\calU} \qquad & \text{if }\beta\in \bbR \setminus \{0\}.
                \end{cases}
    \]

    If $i\beta \in i\bbR$ is an approximate eigenvalue of $A$ that is not in the point spectrum of $A$, then for $\lambda \in \bbC$ with $\re \lambda>0$, the $\lambda$-independent set
    \[
        K:= \left\{x^{\calU} \in \Fix S^\beta: (\lambda-i\beta)\Res(\lambda,A)^{\calU}x^{\calU}=x^{\calU}\right\}
    \]
    is infinite-dimensional.
\end{proposition}

\begin{proof}
    Let $(x_n)$ be a normalized approximate eigenvector of $A$ corresponding to $i\beta$. Since $i\beta$ is not an eigenvalue of $A$ but accumulation points of the approximate eigenvectors are eigenvectors, so $(x_n)$ can not have accumulation points. Replacing $(x_n)$ with a subsequence, we can therefore find an $\epsilon>0$ such that $\norm{x_n-x_m}\ge \epsilon$ whenever $n\ne m$.

    On the other hand, since $(x_n)$ is an approximate eigenvector of $A$ corresponding to $i\beta$,
    Proposition~\ref{prop:eigenvector-fixed-space} allows us to further replace $(x_n)$ by a subsequence such that $(x_{n})^{\calU}\in K$. As a result, the unit ball of $K$ contains the sequence $\left(y_k^{\calU}\right)$
    given by $y_k^{\calU} := (x_{n+k})^{\calU}$  that satisfies
    \[
        \norm{ y_k^{\calU}- y_m^{\calU} } = \lim_{\calU} \norm{ x_{n+k} - x_{n+m}  } \ge \epsilon
    \]
    whenever $k\ne m$. This means that the unit ball of $K$ is not relatively compact and so $K$ must be infinite-dimensional.
\end{proof}

Since resolvent operators commute with the semigroup,
Propositions~\ref{prop:eigenvector-fixed-space} and~\ref{prop:infinite-dimensional-eigenspace-resolvent} hint that if the semigroup is eventually positive, then $\left\{\Res(\lambda, A)^{\calU}\restrict{\Fix S^\beta}\right\}_{\re\lambda>0}$ could be the positive pseudo-resolvent that we are seeking. However, generally, $\Fix S^\beta$ is not a Banach lattice which renders the known results on positive pseudo-resolvents inapplicable. Nevertheless, we know from \cite[Proposition~7.2]{Glueck2017} that the fixed space of the dual of a positive power-bounded operator is always a Banach lattice (post-equivalent renorming). This indicates that working with biduals might be the right approach and indeed this is the content of Theorem~\ref{thm:ultrapower-bidual}.

Recall that a $C_0$-semigroup $(e^{tA})_{t \ge 0}$ on a Banach lattice $E$ with $\spb(A)=0$ is called \emph{uniformly asymptotically positive} if $(e^{tA})_{t \ge 0}$ is bounded and for each $\epsilon>0$, there exists $t_0\ge 0$ such that
\[
    \dist( e^{tA}f, E_+)\le \epsilon \norm{f}\quad \text{for all }t\ge t_0 \text{ and } 0\le f \in E;
\]
if the spectral bound is not $-\infty$ but different from $0$, then it is said to be \emph{uniformly asymptotically positive} if the above is true for the rescaled semigroup $\left(e^{t(A-\spb(A))}\right)_{t \ge 0}$.

\begin{theorem}
    \label{thm:ultrapower-bidual}
    Let $(e^{tA})_{t \ge 0}$ be a uniformly asymptotically positive $C_0$-semigroup on a complex Banach lattice $E$, 
    let $\calU$ be a  free ultrafilter on $\bbN$, and let $c \in \bbR\setminus \{0\}$. 
    The bounded operators
    \[
        S:=\left(\left(e^{c n (A-\spb(A))}\right)_n\right)^{\calU} \in \calL\left(E^{\calU}\right)
        \quad \text{ and } \quad 
        R(\argument):=\left(\Res(\argument, A)^{\calU}\right)'' \in \calL\left( \left(E^{\calU}\right)''\right)
    \]
    satisfy the following assertions:
    \begin{enumerate}[\upshape (a)]
        \item The operator $S$ is positive, power bounded, and the fixed space $F:=\Fix S''$ is a Banach lattice with respect to an equivalent norm.
        \item The operator family $\left\{ R(\lambda)\right\}_{\re \lambda>\spb(A)}$ is $F$-invariant and $\left\{ R(\lambda)\restrict{F}\right\}_{\re \lambda>\spb(A)} $ is a positive pseudo-resolvent on $F$.
    \end{enumerate}
\end{theorem}

For the proof of Theorem~\ref{thm:ultrapower-bidual}, we need the following simple observation about the positive cone of the ultrapower:

\begin{lemma}
    \label{lem:ultrapower-positive}
    Let $E$ be a Banach lattice and let $\calU$ be a  free ultrafilter on $\bbN$. Let $x^{\calU}=(x_n)^{\calU} \in E^{\calU}$ such that $\dist(x_n, E_+)\to 0$ as $n\to\infty$, then $x^{\calU} \ge 0$.
\end{lemma}

\begin{proof}
    By assumption, we can find a strictly increasing sequence $(n_k)$ in $\bbN$ such that $\dist(x_{n_k}, E_+)<\frac1k$ for all $k\in \bbN$. Thus for each $k\in \bbN$, there exists $0\le y_k\in E$ such that $\norm{x_{n_k}-y_k}<\frac1k$. So, for each $\epsilon>0$, the set
    \[
        A_{\epsilon}:= \{k\in \bbN: \norm{x_{n_k}-y_k}<\epsilon\} \in \calU
    \]
    due to cofiniteness (recall that free ultrafilters are finer than the Fr\'echet filter); this shows that the $(x_{n_k}-y_k)\in c_{\calU}(E)$. Since $(y_k)\subseteq E_+$, it follows that $x^{\calU}\in E^{\calU}_+$.
\end{proof}

\begin{proof}[Proof of Theorem~\ref{thm:ultrapower-bidual}]
    (a)
    Uniform asymptotic positivity of the semigroup guarantees that $S$ is power bounded and due to Lemma~\ref{lem:ultrapower-positive} also positive. In turn, $S''$ is the dual of
    a positive power-bounded operator on a dual Banach lattice $\left(E^{\calU}\right)''$, so we may infer from \cite[Proposition~7.2]{Glueck2017} that $F$ is a Banach lattice with respect to an equivalent norm.

    (b) As the semigroup operators commute with the resolvent of the generator, the operator family $\{R(\lambda)\}_{\re \lambda>\spb(A)}$ leaves $F$ invariant. In particular, the operator family $\left\{ R(\lambda)\restrict{F}\right\}_{\re \lambda>\spb(A)}$ forms a pseudo-resolvent. 
    
    Finally, assume for simplicity that $\spb(A)=0$,
    fix $\lambda >0$, and let $0\le x^{\calU}=(x_n)^{\calU}\in E^{\calU}$. 
    Boundedness of the semigroup allows us to use the Laplace transform representation of the resolvent on the right half-plane, which yields
    \[
        \Res(\lambda,A)e^{c n A}x_n = \int_0^\infty e^{-\lambda t}e^{tA}e^{cnA}x_n\dx t.
    \]
    Therefore, $\Res(\lambda,A)^{\calU}Sx^{\calU}=\int_0^{\infty} e^{-\lambda t}(e^{tA})^{\calU}S x^{\calU}\dx t$, which is positive due to the asymptotic positivity of the semigroup and positivity of $S$. We have thus proved that $\Res(\lambda,A)^{\calU}S$ and in turn $R(\lambda)S''$ is positive for each $\lambda>0$. Consequently, $\left\{ R(\lambda)\restrict{F}\right\}_{\re \lambda>0}$ is a positive pseudo-resolvent. 
\end{proof}

\section{Peripheral (point) spectrum}
    \label{sec:peripheral-spectrum}

The classical Perron-Frobenius theorem tells us that positive matrices have a cyclic peripheral spectrum. Inspired by this, conditions for cyclicity of the peripheral spectrum are omnipresent for positive operators  \cite{Glueck2016a,Glueck2018}, uniformly eventually positive operators \cite[Theorem~7.1]{Glueck2017}, and positive semigroups \cite[Section~C-III.2]{Nagel1986}. For the case of eventually positive semigroups, a positive result for the cyclicity of the peripheral point spectrum has been known due to Glück \cite[Theorem~6.3.2]{Glueck2016}. In this section, we give conditions that ensure cyclicity of the peripheral spectrum generalising the analogous results for positive (irreducible) semigroups \cite[Theorems~C-III-3.10 and~C-III-3.12]{Nagel1986}.

\subsection*{Peripheral spectrum of persistently irreducible semigroups}

The peripheral spectrum of a positive irreducible semigroup on a Banach lattice is contained in the point spectrum provided that the spectral bound is a pole of the resolvent; see \cite[Lemma~14.16]{BatkaiKramarRhandi2017} for the statement and proof of \cite[Theorem~C-III-3.12]{Nagel1986} for the argument. Leveraging the results of Section~\ref{sec:ultrapowers}, we generalise this for semigroups that are bounded, uniformly eventually positive, and persistently irreducible. Recall that the \emph{peripheral spectrum} of a closed linear operator $A$ is defined as
\[
    \perSpec(A):= \{ \lambda\in \spec(A): \re \lambda=\spb(A)\}.
\]
Moreover, for a vector $u$ in a Banach space $E$ and a functional $\varphi \in E'$, the (at most) rank-one operator $u\otimes \varphi \in \calL(E)$ is defined as $f\mapsto \duality{\varphi}{f}u$.

\begin{proposition}
    \label{prop:peripheral-in-point}
    Let $E$ be a complex Banach lattice and let $(e^{tA})_{t \ge 0}$ be a  real, uniformly eventually positive, and persistently irreducible semigroup on $E$.

    If $\spb(A)$ is a pole of the resolvent $\Res(\argument,A)$ and the rescaled semigroup $\left(e^{t(A-\spb(A))}\right)_{t \ge 0}$ is bounded, then every element of $\perSpec(A)$ is an eigenvalue of $A$.
\end{proposition}

\begin{proof}
    Without loss of generality, assume that $\spb(A)=0$. Since $(e^{tA})_{t \ge 0}$ is  real, individually eventually positive, and persistently irreducible, we can apply \cite[Corollary~4.4(c)]{AroraGlueck2024} to get that the order of $\spb(A)$ as a pole of the resolvent $\Res(\argument,A)$ must be one and the corresponding spectral projection is given by $P= u\otimes \varphi$ for some quasi-interior point $u$ of $E_+$ and some a strictly positive functional $\varphi\in E'$.

    Let $i\beta\in i\bbR$ be a spectral value of $A$ and $U$ be a free ultrafilter on $\bbN$. On the space $E^{\calU}$, consider the bounded operators
    \[
        S^\beta:=\begin{cases}
                    \left(\left(e^{ nA}\right)_n\right)^{\calU} \qquad & \text{if }\beta=0\\
                    \left(\left(e^{\frac{2\pi n}{\beta}A}\right)_n\right)^{\calU} \qquad & \text{if }\beta\ne 0.
                \end{cases}
    \]
    Uniform eventual positivity and boundedness of the semigroup allow us to infer from Theorem~\ref{thm:ultrapower-bidual} that $S^\beta$ is a positive operator, $F:=\Fix \left(S^{\beta}\right)''$ is a Banach lattice with respect to an equivalent norm,
    and the operator family $\left\{ \left(\Res(\lambda, A)^{\calU}\right)''\restrict{F}\right\}_{\re \lambda>0} $ forms a positive pseudo-resolvent on $F$. 
    
    Let us show that $\left(P^{\calU}\right)''$ maps into $F$.  Since, $0$ is a pole of the resolvent $\Res(\argument, A)$ with corresponding projection $P$, the element $u$ lies in the fixed space of the semigroup. Additionally, it is easy to see that $P^{\calU}=u^{\calU} \otimes \varphi^{\calU}$. This means that, $\left(P^{\calU}\right)''$ must indeed map into $F$.
    From now on, we exclusively work on the Banach lattice $F$ and for simplicity, fix the notation
    \[
        R_F(\argument):=\left(\Res(\argument, A)^{\calU}\right)''\restrict{F}
        \quad  \text{ and } \quad
        P_F := \left(P^{\calU}\right)''\restrict{F} \in \calL(F).
    \]
    Note that, $\dim \Ima P^{\calU}=1$ implies that $\dim \Ima P_F=1$.

    We want to show that $i\beta$ is an eigenvalue of $A$. If not, then we may deduce from Proposition~\ref{prop:infinite-dimensional-eigenspace-resolvent} that for each $\lambda \in \bbC$ with $\re \lambda>0$, the $\lambda$-independent eigenspace
    \[
        K:= \left\{x \in F: (\lambda-i\beta)R_F(\lambda)x=x\right\}
    \]
    is infinite-dimensional; cf. Remark~\ref{rem:eigenvector-fixed-space}(a).
    Next, we consider the closed ideal
    \[
        I:= \left\{  x \in F : P_F\modulus{ x}=0\right\}.
    \]
    For each $x \in K$, we obtain from~\eqref{eq:positive-pseudo-resolvent} the estimate
    \begin{equation}
        \label{eq:ideal-element}
          \modulus{x}=\modulus{r R_F(r+i\beta)x}\le r R_F(r)\modulus{x} \xrightarrow{r\downarrow 0} P_F\modulus{x},
    \end{equation}
    which yields $K\cap I= \{0\}$; the inequality above uses positivity of the pseudo-resolvent $\left\{ R_F(\lambda)\right\}_{\re \lambda>0}$. As a result, the subspace $K/I$ of the quotient space $F/I$ must also be infinite-dimensional. 
    
    On the other hand, for each $x \in K$, the estimate~\eqref{eq:ideal-element} implies that 
    $P_F\modulus{x}-\modulus{x}$ is positive, and hence an element of $I$. In other words, for each $\widehat x \in K/I$, we have $\modulus{\widehat x} \in \Ima \widehat{P_F}$; where $\widehat{P_F}$ denotes the operator induced by $P_F$ on $F/I$. Consequently, \cite[Lemma~C-III-3.11]{Nagel1986} tells us that $\dim K/I \le \dim \Ima \widehat{P_F}$, so $\Ima \widehat{P_F}$ is infinite-dimensional as well, contradicting that $\dim \Ima P_F=1$.
\end{proof}

\begin{remarks}
    (a) Unlike \cite[Lemma~14.16]{BatkaiKramarRhandi2017}, we needed to assume boundedness of the (rescaled) semigroup in Proposition~\ref{prop:peripheral-in-point} in order to employ the ultrapower techniques of Section~\ref{sec:ultrapowers}.

    (b) A careful reader might observe that the proof of Proposition~\ref{prop:peripheral-in-point} used two kinds of eventual positivity:~uniform asymptotic positivity to apply Theorem~\ref{thm:ultrapower-bidual} and individual eventual positivity to employ \cite[Corollary~4.4(c)]{AroraGlueck2024}. 
    This naturally poses the question of whether uniform asymptotic and individual eventual positivity together imply uniform eventual positivity.
    
    Surprisingly, this is not the case. In fact, the first given examples of semigroups that are individually but not uniformly eventually positive \cite[Examples~5.7 and~5.8]{DanersGlueckKennedy2016a} are indeed uniformly asymptotically positive due to \cite[Theorem~7.3.6]{Glueck2016}.
    
    In summary, Proposition~\ref{prop:peripheral-in-point} is valid for semigroups that are merely individually eventually and uniformly asymptotically positive.
\end{remarks}

In case the spectral bound is a pole of the resolvent, the peripheral spectrum of a positive (persistently) irreducible semigroup on a Banach lattice not just consists only of eigenvalues but it is even true that the peripheral spectrum is cyclic (see \cite[Theorem~14.17]{BatkaiKramarRhandi2017} or \cite[Theorem~C-III-3.12]{Nagel1986}). The same remains true in the situation of Proposition~\ref{prop:peripheral-in-point}. In what follows, we write $\pntSpec(A)$ for the point spectrum of $A$ and use the notation
\[
    \perpntSpec(A):= \perSpec(A)\cap \pntSpec(A) = \pntSpec(A) \cap (\spb(A)+i\bbR).
\]

\begin{theorem}
    \label{thm:cyclic-peripheral-point}
    Let $(e^{tA})_{t \ge 0}$ be a real, uniformly eventually positive, and persistently irreducible semigroup on a complex Banach lattice $E$ such that $\spec(A)\ne\varnothing$.

    If $\spb(A)$ is a pole of the resolvent $\Res(\argument,A)$, then there exists $\alpha\ge 0$ such that $\perpntSpec(A)=\spb(A)+i\alpha \bbZ$.
     Moreover, every pole in $\perSpec(A)$ is a simple pole.
\end{theorem}

For the proof of Theorem~\ref{thm:cyclic-peripheral-point}, we prove a generalisation of \cite[Proposition~14.15(c)]{BatkaiKramarRhandi2017} that needs the concept of signum operators on Banach lattice.
Let $K$ be a compact Hausdorff space and $f\in C(K)$ be a function that vanishes nowhere. In particular, $\modulus{f}$ is strictly positive and so $\sgn f:=f \modulus{f}^{-1}$ is well-defined. This gives rise to an operator $S_f : C(K)\to C(K)$ defined by $g\mapsto (\sgn f)g$ that satisfies
\begin{equation}
    \label{eq:signum-operator}
    S_f \bar{f} = \modulus f \qquad\text{and}\qquad \modulus{S_fg}\le \modulus{g}. 
\end{equation}

In fact, for a general Banach lattice $E$, if $f\in E$ such that $\modulus{f} $ is a quasi-interior point of $E_+$, then by Kakutani representation theorem, one can uniquely obtain the so-called \emph{signum operator} $S_f : E\to E$ with respect to $f$ that is invertible and satisfies the properties in~\eqref{eq:signum-operator}. We refer to \cite[Page~227]{BatkaiKramarRhandi2017} for the construction.

\begin{lemma}
    \label{lem:cyclic-peripheral-point}
    Let $E$ be a complex Banach lattice and let $(e^{tA})_{t \ge 0}$ be a  real, uniformly eventually positive, and persistently irreducible semigroup on $E$.

    If $\spb(A)=0$ and $\ker A'$ contains a strictly positive functional, then $\pntSpec(A) \cap i\bbR$ is either empty or an additive subgroup of $i\mathbb R$.
\end{lemma}

\begin{proof}
    By assumption, there exists a strictly positive functional $\varphi \in \ker A'$ and $t_0\ge 0$ such that $e^{tA}\ge 0$ for all $t\ge t_0$. Let $\alpha \in \mathbb R$ and $0\ne f\in E$ be such that $Af=i\alpha f$. Then from the spectral inclusion theorem \cite[Corollary~9.32]{BatkaiKramarRhandi2017}, we get $e^{tA}f=e^{i\alpha t}f$ for all $t\ge 0$. In particular,
    \[
        \modulus{f}=\modulus{e^{i\alpha t}f} = \modulus{e^{tA}f} \le e^{tA}\modulus{f}
    \]
    for all $t\ge t_0$. Hence, for each $t\ge t_0$, the vector $e^{tA}\modulus{f}-\modulus{f}$ is a positive element in the kernel of $\varphi$; here we have used that $e^{tA'}\varphi=\varphi$ for all $t\ge 0$.
    It follows from the strict positivity of $\varphi$ that $e^{tA}\modulus{f}=\modulus{f}$ for all $t\ge t_0$. As a result,
    \begin{equation}
        \label{eq:modulus-fixed-point}
        e^{tA}\modulus{f} = e^{tA} e^{t_0 A} \modulus{f} = e^{(t+t_0)A} \modulus{f} = \modulus{f}
    \end{equation}
    for all $t\ge 0$ and so $\modulus{f}\in \ker A$. 
    Now, for every eventually positive, persistently irreducible semigroup on Banach lattice, we know from \cite[Proposition~4.1]{AroraGlueck2024} that every positive non-zero eigenvector of the generator is a quasi-interior point. 
    Thus, $\modulus f$ is a quasi-interior point of $E_+$ and so the signum operator $S_f$ is well-defined. So by a generalisation of Wielandt's lemma \cite[Lemma~14.14]{BatkaiKramarRhandi2017}, we obtain
    \[
        e^{tA} = S_f^{-1} e^{t(A-i\alpha)} S_f \qquad (t\ge t_0).
    \]

    Finally, let $i\alpha, i\beta \in \pntSpec(A)$. Then there exists $0\ne f,g\in E$ such that $Af=i\alpha f$ and $Ag=i\beta g$. From what we have shown above,
    \[
        e^{t(A-i(\alpha+\beta))} = S_g e^{t(A-i\alpha)} S_g^{-1} = S_g S_f e^{tA} S_f^{-1} S_g^{-1}
    \]
    for all $t\ge t_0$. In particular, using~\eqref{eq:modulus-fixed-point}, we obtain
    \[
        e^{t(A-i(\alpha+\beta))} S_g S_f \modulus{f} =S_g S_f e^{tA} \modulus{f}= S_g S_f \modulus{f}
    \]
    for all $t\ge t_0$. By invertibility of the signum operator, it follows that $0\ne S_g S_f \modulus{f} \in \ker (A-i(\alpha+\beta))$ and so $i(\alpha+\beta)\in \pntSpec(A)$.
\end{proof}

We point out that uniform rather than individual eventual positivity of the semigroup was only needed in the proof above to employ \cite[Lemma~14.14]{BatkaiKramarRhandi2017}. The preceding arguments also work for the individual case, cf. \cite[Lemma~4.2]{DanersGlueck2017}.

\begin{proof}[Proof of Theorem~\ref{thm:cyclic-peripheral-point}]
    Replacing $A$ with $A-\spb(A)$, we may assume that $\spb(A)=0$.
    Since $0$ is a pole of the resolvent and the semigroup is individually eventually positive and persistently irreducible, so $\ker A'$ contains a strictly positive functional by \cite[Corollary~4.4(c)]{AroraGlueck2024}.  
    
    Lemma~\ref{lem:cyclic-peripheral-point} thus implies that $\pntSpec(A) \cap i\bbR$ is an additive subgroup of $i\mathbb R$ (it is non-empty because it contains $0$). For this reason, it is either of the form $i\alpha \bbZ$ or is dense. The latter is not possible because
    $\pntSpec(A) \cap i\bbR$ is closed and contains an isolated point. Consequently, there must exist $\alpha\ge 0$ such that $\pntSpec(A) \cap i\bbR=i\alpha \bbZ$. 

    Finally, individual eventual positivity and persistent irreducibility of the semigroup implies by \cite[Corollary~4.4(c)]{AroraGlueck2024} that because $\spb(A)\ne-\infty$ is a pole of the resolvent $\Res(\argument,A)$, it must be a simple pole. In turn, every pole in $\perSpec(A)$ must be a simple pole by \cite[Proposition~6.3.1]{Glueck2016}.
\end{proof}

\begin{remark}
    Since in the proof of Lemma~\ref{lem:cyclic-peripheral-point}, we obtain the equality $e^{tA} = S_f^{-1} e^{t(A-i\alpha)} S_f$ only for large $t$, we are -- unlike \cite[Theorem~C-III-3.12]{Nagel1986} -- unable to conclude in Theorem~\ref{thm:cyclic-peripheral-point} that every element of $\perSpec(A)$ is a pole of the resolvent.
\end{remark}

An immediate consequence of Proposition~\ref{prop:peripheral-in-point} and Theorem~\ref{thm:cyclic-peripheral-point} is the following:

\begin{corollary}
    Let $(e^{tA})_{t \ge 0}$ be a real, uniformly eventually positive, and persistently irreducible semigroup on a complex Banach lattice $E$ such that $\spec(A)\ne\varnothing$.

    If $\spb(A)$ is a pole of the resolvent $\Res(\argument,A)$ and the rescaled semigroup $\left(e^{t(A-\spb(A))}\right)_{t \ge 0}$ is bounded, then there exists $\alpha\ge 0$ such that $\perSpec(A)=\spb(A)+i\alpha \bbZ$.
\end{corollary}

\subsection*{Peripheral spectrum of asymptotically positive semigroups}

If $(e^{tA})_{t \ge 0}$ is a positive and bounded $C_0$-semigroup on a Banach lattice, then $\spb(A)=0$ implies that the peripheral spectrum of $A$ is cyclic \cite[Theorem~C-III-2.10 and Proposition~C-III-2.9(a)]{Nagel1986}. We generalise this result for uniform asymptotic positivity:

\begin{theorem}
    \label{thm:cyclicity-peripheral}
    If $(e^{tA})_{t \ge 0}$ is a uniformly asymptotically positive $C_0$-semigroup on a complex Banach lattice $E$, then
    the peripheral spectrum of $A$ is cyclic, i.e., if $\spb(A)+i\beta \in \spec(A)$, then $\spb(A)+in\beta \in \spec(A)$ for all $n\in \bbZ$.
\end{theorem}

\begin{proof}
    Without loss of generality, we assume that $\spb(A)=0$.
    Let $\beta\ne 0$ be such that $i\beta$ is a spectral value of $A$ and
    let $\calU$ be a free ultrafilter on $\bbN$. 
    We work with bi-adjoints for which we fix the notations:
    \[
        G:=\left(E^{\calU}\right)'', \quad S:=\left(\left(\left(e^{\frac{2\pi n}{\beta}A}\right)_n\right)^{\calU}\right)'', \quad \text{ and } \quad R(\argument)=\left(R(\argument, A)^{\calU}\right)''.
    \]
    Theorem~\ref{thm:ultrapower-bidual} tells us that $S$ is positive, $F:=\Fix  S$ is a Banach lattice with respect to an equivalent norm, and $\left\{ R(\lambda)\restrict{F}\right\}_{\re \lambda>0} $ is a positive pseudo-resolvent on $F$.
    On the other hand, Proposition~\ref{prop:eigenvector-fixed-space} yields a non-zero vector $ x \in F$ such that
    \begin{equation}
        \label{eq:resolvent-eigenvector}
        \lambda R(\lambda+i\beta)  x =  x \quad \text{ whenever }\quad\re \lambda>0
    \end{equation}
    (the semigroup is bounded due to uniform asymptotic positivity and $\spb(A)=0$).
    Now, we divide the proof into three steps and argue as in the proof of \cite[Theorem~C-III-2.10]{Nagel1986}.
    
    \emph{Step 1: We find a closed ideal of $F$ such that the image of $x$ under the corresponding quotient map is non-zero}. First pick $\varphi\in F'$ such that $\duality{\varphi}{ x}\ne 0$. Boundedness of the semigroup guarantees norm-boundedness of the operator family $\{r R(r)' \modulus{\varphi} : r>0 \}$ in $F'$. So, by the Banach-Alaoglu theorem, it has a weak${}^*$-convergent subnet  $\big(r_j R(r_j)' \modulus{\varphi} \big)$ with limit, say $\psi \in F'$, as $r_j\to 0$. 
    Observe that, positivity of the pseudo-resolvent,~\eqref{eq:positive-pseudo-resolvent}, and~\eqref{eq:resolvent-eigenvector} together imply 
    \begin{equation}
        \label{eq:resolvent-positive-eigenvector}
        \modulus{x} = \modulus{r R(r+i\beta) x} \le r R(r)\modulus{x}\qquad (r>0).
    \end{equation}
    Consequently,
    \[
        \modulus{\duality{\varphi}{ x}}  \le \duality{\modulus{\varphi}}{r R(r)\modulus{x}  } = \duality{r R(r)'\modulus{\varphi} }{\modulus x}\qquad (r>0).
    \]
    From the weak${}^*$-convergence we now infer that 
    $
        \duality{\psi}{\modulus x} \ge \modulus{\duality{\varphi}{ x}} > 0
    $
    which means that the ideal
    \[
        I := \{ y \in F: \duality{\psi}{\modulus y}=0\}
    \]
    does not contain $x$. Denoting the quotient space $F/I$ by $\widehat F$, we have shown that $\widehat x := x+I$ is a non-zero vector in $\widehat F$.
    
    \emph{Step 2: We show that $\left\{ R(\lambda)\restrict{F}\right\}_{\re \lambda>0} $ induces a positive pseudo-resolvent on $\widehat F$}.
    For fixed $\lambda \in \bbC$ with $\re\lambda>0$, the resolvent identity gives
    \begin{align*}
        \big(1-\lambda R(\lambda)'\big)r_j R(r_j)'\modulus{\varphi}
            & = r_j R(r_j)'\modulus{\varphi} - r_j \lambda ( r_j-\lambda)^{-1} \big(R(\lambda)'-R(r_j)'\big)\modulus{\varphi}\\
            & = r_j  ( r_j-\lambda)^{-1} \big(r_j R(r_j)'\modulus{\varphi}-\lambda R(\lambda)'\modulus{\varphi}\big).
    \end{align*}
    Norm-boundedness of $\{r R(r)' \modulus{\varphi} : r>0 \}$ and weak${}^*$-convergence of $\big(r_j R(r_j)' \modulus{\varphi} \big)$ therefore imply that
    \begin{equation}
        \label{eq:resolvent-zero-eigenvalue}
        \lambda R(\lambda)'\psi = \psi.
    \end{equation} 
    In particular, for each $y\in F$, 
    \[
        \duality{\psi}{R(\lambda) \modulus{y} } = \duality{R(\lambda)'\psi }{\modulus y} = \lambda^{-1} \duality{\psi}{\modulus y}.
    \]
    Whence, we have shown that $I$ is invariant under $\left\{ R(\lambda)\restrict{F}\right\}_{\re \lambda>0} $ and so the latter induces a positive pseudo-resolvent $\left\{\widehat R(\lambda)\right\}_{\re \lambda>0} $ on $\widehat F$.

    \emph{Step 3: We show the required assertion that $in\beta \in \spec(A)$ for all $n\in \bbZ$}.
    Fix $\lambda \in \bbC$ with $\re\lambda>0$. We claim that $\lambda \widehat R(\lambda)\modulus{ \widehat x}=\modulus{\widehat  x}$. By~\eqref{eq:pseudo-resolvent-eigenvalue-sufficient}, it suffices to show the equality for the particular case $\lambda=r>0$. 
    In this case,~\eqref{eq:resolvent-zero-eigenvalue} yields
    \[
        \duality{\psi}{r R(r)\modulus x-\modulus x  }= \duality{rR(r)'\psi-\psi}{\modulus x} =0
    \]
    and so -- keeping~\eqref{eq:resolvent-positive-eigenvector} in mind -- we get $rR(r)\modulus x-\modulus x \in I$. This means that $r\widehat R(r)\modulus{\widehat  x} =\modulus{\widehat x}$, as desired. Furthermore~\eqref{eq:resolvent-eigenvector} implies that $\lambda \widehat R(\lambda+i\beta)\widehat x=\widehat x$. Next, let $S_{\widehat x}$ and $S_{\overline{\widehat x}}$ be the signum operator on the principal ideal ${\widehat F}_{\modulus{\widehat x}}$ (equipped with gauge norm) and for $n\in \bbZ$, recursively set
    \[
       \widehat x^{[n]} :=  \begin{cases}
                                \modulus{\widehat x},\qquad & n=0\\
                                S_{\widehat x} \widehat x^{[n-1]}, \qquad & n>0\\
                                S_{\overline{\widehat x}} \widehat x^{[n+1]}, \qquad & n<0;
                            \end{cases} 
    \]
    see \cite[C-III-Definition~2.1]{Nagel1986}.
    Because $\left\{\widehat R(\lambda)\right\}_{\re \lambda>0} $ is a positive pseudo-resolvent, $\lambda \widehat R(\lambda)\modulus{ \widehat x}=\modulus{\widehat  x}$, and $\lambda \widehat R(\lambda+i\beta)\widehat x=\widehat x$ for all $\lambda \in \bbC$ with $\re\lambda>0$, we must, according to \cite[Proposition~C-III-2.7]{Nagel1986}, also have
    \[
        \lambda \widehat R(\lambda+in\beta)\widehat x^{[n]}=\widehat x^{[n]}, \qquad n\in\bbZ.
    \]
    So, for each $r>0$ and $n\in \bbZ$, we have
    \[
        \norm{\Res(r+ in\beta,A)}=\norm{\Res(r+ in\beta,A)^{\calU}} = \norm{R(r+in\beta)} = \norm{\widehat R(r+in\beta)} \ge \frac1r.
    \]
    Consequently, $in\beta\in \spec(A)$ for all $n\in \bbZ$.
\end{proof}

Theorem~\ref{thm:cyclicity-peripheral} allows us to know the exact description of the peripheral spectrum if we additionally know that the peripheral spectrum is bounded. Note that the spectral bound of a closed operator $A$ is said to be a \emph{dominant spectral value} if
$
    \perSpec(A)=\{\spb(A)\}.
$
Further, we recall that a $C_0$-semigroup on a Banach space is called \emph{norm continuous at infinity} if either $\gbd(A)=-\infty$ or
\[
    \lim_{t\to\infty} \limsup_{s\to 0} \norm{ e^{t (A-\gbd(A)) } - e^{(t+s)(A- \gbd(A))}   }=0.
\]
Semigroups that are norm continuous at infinity were studied first in \cite{MartinezMazon1996}. Their significance is theoretical because even though they strictly contain the class of eventually norm continuous semigroups, they retain various spectral and asymptotic properties (see \cite{MartinezMazon1996} and \cite[Lemma~8.2.4]{Arora2023}).

\begin{corollary}
    \label{cor:dominant-spectral-value}
    On a complex Banach lattice $E$, if $(e^{tA})_{t \ge 0}$ be a real $C_0$-semigroup that is norm continuous at infinity and uniformly asymptotically positive,
    then $\spb(A)$ is a dominant spectral value of $A$.
\end{corollary}

\begin{proof}
    We know from Theorem~\ref{thm:cyclicity-peripheral} that the peripheral spectrum of $A$ is cyclic. On the other hand, due to the norm-continuity of the semigroup, the peripheral spectrum of $A$ is bounded \cite[Theorem~1.9]{MartinezMazon1996}. It follows that $\spb(A)$ is a dominant spectral value of $A$.
\end{proof}

\begin{remark}
    Previously, Corollary~\ref{cor:dominant-spectral-value} was only available for uniformly \emph{eventually} positive semigroups \cite[Lemma~3.3]{AroraGlueck2021a}. A careful glance at the proofs of convergence results \cite[Theorems~3.1 and~5.1]{AroraGlueck2021a} tells us that the uniform eventual positivity assumption was only needed in order to apply \cite[Lemma~3.3]{AroraGlueck2021a}. Therefore, Corollary~\ref{cor:dominant-spectral-value} tells us that in \cite[Theorems~3.1 and~5.1 and Corollary~3.2]{AroraGlueck2021a}, one can weaken the uniformly eventual positivity assumption to uniform asymptotic positivity.
\end{remark}

\section{Peripheral spectrum under asymptotic domination}
    \label{sec:domination}

A phenomenon closely related to positive $C_0$-semigroups is that of domination of semigroups. In this case, the spectral properties inherited by the dominated semigroup have been investigated in \cite{AndreuMazon1989}, \cite[Corollary~4.5]{Caselles1987}, and \cite[Section~6]{RabigerWolff2000}. Our main result in this section extends \cite[Corollary~6.2]{RabigerWolff2000} for eventually dominated semigroups, a notion introduced in \cite{GlueckMugnolo2021} and expanded upon in \cite{AroraGlueck2023b}. Actually, we only need the weaker notion of asymptotic domination:

\begin{definition}
    Let $(e^{tA})_{t \ge 0}$ and $(e^{tB})_{t \ge 0}$ be $C_0$-semigroups on a complex Banach lattice $E$ such that the rescaled semigroup $(e^{t(B-\spb(B))})_{t \ge 0}$ is bounded.
    We say that $(e^{tB})_{t \ge 0}$ \emph{uniformly asymptotically dominates} $(e^{tA})_{t \ge 0}$ if for each $\epsilon>0$, there exists $t_0\ge 0$ such that
    \[
        \dist\left( e^{t(B-\spb(B))}f - e^{t(A-\spb(B))} f    , E_+\right) \le \epsilon \norm{f}  
    \]
    for all $f\in E_+$ and all $t\ge t_0$.
\end{definition}

If $A$ and $B$ are generators of positive $C_0$-semigroups such that the semigroup generated by $A$ is dominated by the semigroup generated by $B$, then it is easy to see that $\spb(A) \le \spb(B)$; see \cite[Lemma~C-II-4.10]{Nagel1986}. The same remains true if the domination is merely asymptotic:

\begin{proposition}
    \label{prop:spectral-bound-under-domination}
    Let $(e^{tA})_{t \ge 0}$ and $(e^{tB})_{t \ge 0}$ be real $C_0$-semigroups on a complex Banach lattice $E$ such that the rescaled semigroup $(e^{t(B-\spb(B))})_{t \ge 0}$ is bounded
    and $(e^{tB})_{t \ge 0}$ uniformly asymptotically dominates $(e^{tA})_{t \ge 0}$.

    Then $(e^{t(A-\spb(B))})_{t \ge 0}$ is also bounded and in particular, $\spb(A)\le \spb(B)$.
\end{proposition}

\begin{proof}
    We borrow an argument from the proof of \cite[Proposition~2.8]{GlueckWolff2019}.
    Fix $f\in E_+$. Then for each $t\ge 0$, there exists $p_f(t)\in E_+$ such that
    \[
        \norm{ e^{t(B-\spb(B))}f - e^{t(A-\spb(B))} f - p_f(t) } \le \dist\left( e^{t(B-\spb(B))} f - e^{t(A-\spb(B))} f    , E_+\right) + \frac{1}{2^t}.
    \]
    Let $r_f(t)$ denote the negative part of  $e^{t(B-\spb(B))}f - e^{t(A-\spb(B))} f - p_f(t)$, then 
    \begin{align*}
        \norm{r_f(t)} &\le\norm{ e^{t(B-\spb(B))}f - e^{t(A-\spb(B))} f - p_f(t) }\\
                      &\le   \dist\left( e^{t(B-\spb(B))} f - e^{t(A-\spb(B))} f    , E_+\right) + \frac{1}{2^t}.
    \end{align*}
    Uniform asymptotic domination of the semigroup in particular implies that  
    \begin{equation}
        \label{eq:individual-asymptotic-domination}
            \lim_{t\to \infty} \dist\left( e^{t(B-\spb(B))} f - e^{t(A-\spb(B))} f    , E_+\right)=0.
    \end{equation}
    It follows that $\lim_{t\to \infty}r_f(t)=0$ and
    \[
        e^{t(B-\spb(B))} f + r_f(t) \ge e^{t(A-\spb(B))} f+p_f(t) \ge e^{t(A-\spb(B))} f
    \]
    for all $t\ge 0$. In turn, the orbit $(e^{t(A-\spb(B))}f)_{t \ge 0}$ is bounded. Since the positive cone of a Banach lattice is generating, the uniform boundedness principle yields that $(e^{t(A-\spb(B))})_{t \ge 0}$ is bounded. Consequently, $\spb(A) \le \gbd(A)\le \spb(B)$.
\end{proof}

\begin{remark}
    The limit in~\eqref{eq:individual-asymptotic-domination} is the canonical choice to define the \emph{individual asymptotic domination} of semigroups in the spirit of \cite[Definition~8.1(a)]{DanersGlueckKennedy2016b} and in this case Proposition~\ref{prop:spectral-bound-under-domination} remains true.
\end{remark}

We are now ready to state the main results of this section that generalises \cite[Corollary~6.2]{RabigerWolff2000} for asymptotically dominated semigroups.

\begin{theorem}
    \label{thm:spectrum-under-domination}
    Let $(e^{tA})_{t \ge 0}$ and $(e^{tB})_{t \ge 0}$ be real and uniformly asymptotically positive $C_0$-semigroups on a complex Banach lattice $E$. 
    If $(e^{tB})_{t \ge 0}$ uniformly asymptotically dominates $(e^{tA})_{t \ge 0}$,
    then
    \[
        \spec(A) \cap  (\spb(B)+i\bbR) \subseteq \perSpec(B).
    \]
\end{theorem}

Note that then notion of asymptotic domination is well-defined in Theorem~\ref{thm:spectrum-under-domination} because the asymptotic positivity assumption ensures that the semigroup generated by $B-\spb(B)$ is bounded.
 
The proof of \cite[Corollary~6.2]{RabigerWolff2000} relies on the fact the domination of positive semigroups implies the domination of the corresponding resolvents. While we don't have this luxury available, our results of Section~\ref{sec:ultrapowers} and experience of Section~\ref{sec:peripheral-spectrum} suggest that using ultrapower techniques and working with biduals might be a worthwhile endeavour. Before we do this, we outsource the major chunk of our argument to the following lemma whose proof is along the lines of \cite[Theorem~3.2]{RabigerWolff2000}.

\begin{lemma}
    \label{lem:resolvent-under-domination}
    Let $\{R_A(\lambda)\}_{\re \lambda>0}$ and $\{R_B(\lambda)\}_{\re \lambda>0}$ be pseudo-resolvents on a complex Banach lattice $E$. Assume that there exist positive projections $P_A, P_B$ on $E$ satisfying the following:
    \begin{enumerate}[\upshape (a)]
        \item The operators $P_A$ and $P_B$ commute with $R_A$ and $R_B$ respectively.
        \item For each $r>0$, we have $0\le R_A(r)P_A \le R_B(r)P_B$.
        \item The family $\{r R_B(r)P_B: r\in (0,1]\}$ is norm bounded.
    \end{enumerate}
    
    If $i\beta \in i\bbR$ is an eigenvalue of $R_A$ with an eigenvector $x \in \Fix P_A$, then $R_B$ also has a singularity at $i\beta$.
\end{lemma}

\begin{proof}
    We divide the proof into several steps.

    \emph{Step 1: We find a closed ideal to work on quotient spaces}.
    First of all, by assumption~(c), the function
    \[
        p(\argument) := \limsup_{r\downarrow 0} \Big\lVert r R_B(r) P_B \modulus{\argument}\Big\rVert
    \]
    is a continuous lattice seminorm on $E$ and so $I:= \ker p$ is a closed ideal of $E$. Denote the quotient $E/I$ by $\widehat E$ and the induced operators by  $\widehat{R_A}(\argument), \widehat{R_B}(\argument), \widehat{P_A}$, and $\widehat{P_B}$  respectively.
    Using the commutativity assumption~(a), we obtain
    \begin{align*}
        p( R_B(\lambda) P_B y) & = \limsup_{r\downarrow 0}  \big\lVert r R_B(r)P_B \modulus{R_B(\lambda)P_B y}\big\rVert \\
                               & \le \limsup_{r\downarrow 0}  \big\lVert r R_B(r)P_B R_B(\re \lambda) P_B \modulus{y}\big\rVert \\
                               & = \limsup_{r\downarrow 0}  \big\lVert R_B(\re \lambda) P_B r R_B(r)P_B  \modulus{y}\big\rVert \\
                               & \le \norm{R_B(\re \lambda) P_B  } p(y)
    \end{align*}
    for all $y\in E$ and $\re \lambda>0$.
    Therefore, $I$ leaves the operator family $\left\{ R_B(\lambda)P_B\right\}_{\re \lambda>0}$ invariant. The domination property in assumption~(b) even implies that $I$ leaves $\left\{ R_A(\lambda)P_A\right\}_{\re \lambda>0}$ invariant and
    \[
         0 \le \widehat{R_A}(r)\widehat{P_A} \le \widehat{R_B}(r)\widehat{P_B},\qquad r>0.
    \]

    \emph{Step 2: We show that $\widehat x:= x+I$ is an eigenvector of the pseudo-resolvent $\left\{ \widehat{R_A}(\lambda)\right\}_{\re \lambda>0}$}.
    Since $x$ is an eigenvector of $R_A$ with eigenvalue $i\beta$, so $\lambda R_A(\lambda+i\beta)x=x$ whenever $\re\lambda>0$. Assumption~(b), $x\in \Fix P_A$, and positivity of $P_A$ now allow us -- for each $r>0$ --  to estimate
    \begin{equation}
        \label{eq:domination-modulus}
        \modulus{ x} = \modulus{r R_A(r+i\beta)P_Ax } \le r R_A(r) P_A \modulus{x} \le r R_B(r) P_B \modulus{x};
    \end{equation}
    where the first inequality is true due to~\eqref{eq:positive-pseudo-resolvent}. Therefore, $p(x) \ge \norm{x}>0$ and so $\widehat x \ne 0$. Moreover, $\lambda \widehat{R_A}(\lambda+i\beta)\widehat x=\widehat x$ whenever $\re\lambda>0$.

    \emph{Step 3: For fixed $s>0$, we show that $s R_B(s)P_B\modulus{x}-\modulus{x} \in I$}.
    Since $P_B$ is a positive projection that commutes with $R_B$, we deduce from~\eqref{eq:domination-modulus} and the resolvent identity that
    \begin{align*}
        p(s R_B(s)P_B\modulus{x}-P_B\modulus{x}) 
            & = \limsup_{r\downarrow 0} \norm{r R_B(r)P_B \big(s R_B(s)P_B\modulus{x}-\modulus{x}\big)}  \\
            & = \limsup_{r\downarrow 0} \big\lVert rs R_B(r) R_B(s)P_B\modulus{x}-rR_B(r)P_B\modulus{x}\big\rVert  \\
            & = \limsup_{r\downarrow 0} \norm{ r(s-r)^{-1} \big(rR_B(r)P_B\modulus{x} -sR_B(s)P_B\modulus{x}\big)}  \\
            & = 0
    \end{align*}
    by assumption~(c). In turn, $s R_B(s)P_B\modulus{x}-\modulus{x} \in I$.

    \emph{Step 4: We show that $i\beta$ is a singularity of $R_B$}.
    From Step 3, we obtain that 
    \[
        \modulus{\widehat x}= \modulus{s \widehat{R_A}(s+i\beta)\widehat{P_A}\widehat x} \le s \widehat{R_A}(s)\widehat{P_A}\modulus{\widehat  x} \le s \widehat{R_B}(s)\widehat{P_B}\modulus{\widehat x}=\modulus{\widehat x}.
    \]
    In particular, $s \widehat{R_A}(s)\widehat{P_A}\modulus{\widehat  x} = s \widehat{R_B}(s)\widehat{P_B}\modulus{\widehat x}$ and so
    \[
        \modulus{ (s-i\beta) \big( \widehat{R_B}(s)\widehat{P_B}\widehat{x}- \widehat{R_A}(s)\widehat{P_A}\widehat{x}\big)}
        \le \modulus{s-i\beta} s^{-1} \left( s \widehat{R_B}(s)\widehat{P_B} - s \widehat{R_A}(s)\widehat{P_A}\right)\modulus{\widehat x}  =0.
    \]
    We have thus proved that
    \[
        (s-i\beta) \widehat{R_B}(s)\widehat{P_B}\widehat{x}=(s-i\beta) \widehat{R_A}(s)\widehat{P_A}\widehat{x} = \widehat x.
    \]
    Once again, as $P_B$ is a projection that commutes with $R_B$, we obtain that $\widehat x \in \Ima \widehat{P_B}$. This allows us to employ~\eqref{eq:pseudo-resolvent-eigenvalue-sufficient} to infer that $\widehat x$ is an eigenvector of the pseudo-resolvent $\widehat{R_B}(\argument)$ corresponding to the eigenvalue $i\beta$. Therefore, $s \widehat{R_B}(s+i\beta)\widehat{x}=\widehat{x}$. As a result,
    \[
        \norm{R_B(s+i\beta} \ge \norm{\widehat{R_B}(s+i\beta)}\ge \frac1s.
    \] 
    Since $s>0$ was arbitrary, it follows that $i\beta$ is a singularity of $R_B$.
\end{proof}

In order to apply Lemma~\ref{lem:resolvent-under-domination}, we need the following sufficient condition that guarantees the existence of a positive projection onto the fixed space of a positive power-bounded operator.

\begin{lemma}
    \label{lem:projection-existence}
    Let $T$ be a power-bounded weak${}^*$-continuous linear operator on a dual Banach space $X'$. Then there exists a projection $P$ onto the fixed space of $T$ such that the following hold:
    \begin{enumerate}[\upshape (a)]
        \item If $R \in \calL(X)$ is weak${}^*$-continuous and commutes with $T$, then $R$ also commutes with $P$.
        \item If $X$ is a Banach lattice and $T$ is positive, then so is $P$.
    \end{enumerate}
\end{lemma}

Although Lemma~\ref{lem:projection-existence} is likely familiar to experts in operator theory, we could not locate it explicitly stated in this exact form in the literature. For the reader's convenience, we provide the detailed arguments. Our proof uses the concept of universal nets for which we refer to \cite{Willard1970}.

\begin{proof}[Proof of Lemma~\ref{lem:projection-existence}]
    Consider the Cesàro means
    \[
        C_n := \frac1n\sum_{k=0}^{n-1} T^k,\qquad (n\in\bbN).
    \]
    Let $\left(C_{n(\alpha)}\right)_{\alpha \in I}$ be a universal subnet of $(C_n)_{n\in\bbN}$. Then for each $x \in X$, the net $\left(C_{n(\alpha)}x\right)_{\alpha \in I}$ is also universal. Moreover, as $T$ is power-bounded, so by Banach-Alaoglu theorem, this net lies in a weak${}^*$-compact subset of $X'$. Using
    the fact that universal nets in compact spaces are convergent \cite[Theorem~17.4]{Willard1970}, it follows that $C_{n(\alpha)}x$ converges to, say $Px$, in the weak${}^*$-topology. Or equivalently, $C_{{n(\alpha)}}$ converges to $P$ in the weak${}^*$-operator topology. Of course, $P:X\to X$ is linear. 

    We show that $P$ is a projection onto the fixed space $\Fix(T)$. First, note that since $(C_{n(\alpha)})_{\alpha \in I}$ is subnet of $(C_n)_{n\in\bbN}$, so $n:I\to \bbN$ is monotone and $n(I)$ is cofinal in $\bbN$. Thus, the equality
    \[
        (T-I)C_{n(\alpha)}=\frac{1}{n(\alpha)}(T-I) \sum_{k=0}^{n(\alpha)-1} T^k = \frac{1}{n(\alpha)}\left(T^{n(\alpha)} - I\right),
    \]
    and power-boundedness of $T$ implies that $(T-I)C_{\alpha}$ converges to $0$ in the weak${}^*$-operator topology. On the other hand, weak${}^*$-continuity of $T$ implies that $(T-I)C_{\alpha}$ converges to $(T-I)P$ in the weak${}^*$-operator topology. Thus, $(T-I)P=0$. This shows that $\Ima P \subseteq \Fix(T)$. Furthermore, for $x \in \Fix(T)$, we obtain
    \[
        C_{n(\alpha) }x= \frac{1}{n(\alpha)} \sum_{k=0}^{n(\alpha)-1} T^kx=x,
    \]
    so by weak${}^*$-convergence $Px=x$. Thus, $\Ima P = \Fix(T)$ and $P\restrict{\Fix(T)}$ is the identity operator, which is why $P$ is a projection.

    (a) Let $R \in \calL(X)$ be weak${}^*$-continuous operator that commutes with $T$. Then the relation
    \[
        RC_{n(\alpha)} = R\left(\frac{1}{n(\alpha)}\sum_{k=0}^{n(\alpha)-1} T^k\right)= \left(\frac{1}{n(\alpha)}\sum_{k=0}^{n(\alpha)-1} T^k\right) R = C_{n(\alpha)}R
    \]
    along with the weak${}^*$-operator convergence implies that $RP=PR$, as desired.

    (b) This is simply a consequence of the weak${}^*$-closedness of the dual cone $X_+$.
\end{proof}

\begin{proof}[Proof of Theorem~\ref{thm:spectrum-under-domination}]
    Without loss of generality, assume that $\spb(B)=0$ and let $i\beta \in i\bbR$ be a spectral value of $A$. Then Proposition~\ref{prop:spectral-bound-under-domination} implies that $\spb(A)=0$.

    Let $\calU$ be a free ultrafilter on $\bbN$.
    We work with the bi-adjoints on the space $G:=\left(E^{\calU}\right)''$ for which we fix the notations:
    \[
        R_A(\argument)=\left(\Res(\argument, A)^{\calU}\right)''
        \quad  \text{ and } \quad
        S_A^\beta:=  
            \begin{cases}
                    \left(\left(\left(e^{ nA}\right)_n\right)^{\calU}\right)'' \qquad & \text{if }\beta=0\\
                    \left(\left(\left(e^{\frac{2\pi n}{\beta}A}\right)_n\right)^{\calU}\right)'' \qquad & \text{if }\beta\ne 0
            \end{cases}
    \]
    and similarly, $R_B(\argument)$ and $S_B^\beta$.
    Uniform asymptotic positivity allow us to employ Theorem~\ref{thm:ultrapower-bidual} to deduce that $S_A^\beta$ and $S_B^\beta$ are positive, power bounded, and that $F_A:=\Fix  S_A$ and $F_B:=\Fix  S_B$ are both Banach lattices after equivalent renorming. Let $P_A$ and $P_B$ denote the projections onto $F_A$ and $F_B$ respectively, that exist and are positive by Lemma~\ref{lem:projection-existence}.

    In addition, as the semigroup operators commute with the resolvent of the generator, the pseudo-resolvent $\{R_A(\lambda)\}_{\re \lambda>0}$ leaves $F_A$ invariant. Since $R_A$ is weak${}^*$-continuous, we infer from Lemma~\ref{lem:projection-existence}(a) that  $R_A$ commutes with $P_A$. Likewise, $R_B$ commutes with $P_B$.
    Uniform asymptotic domination and positivity of the semigroups also guarantee that $0\le R_A(r)P_A \le R_B(r)P_B$ for all $r>0$; cf. proof of Theorem~\ref{thm:ultrapower-bidual}(b).
    Whence, all assumptions of Lemma~\ref{lem:resolvent-under-domination} are fulfilled; note that condition~(c) is satisfied due to the boundedness of the corresponding semigroup (which holds because they are uniformly asymptotically positive).

    Finally, because $i\beta \in i\bbR$ is a spectral value of $A$, Proposition~\ref{prop:eigenvector-fixed-space} ensures that $i\beta$ is an eigenvalue of $R_A$ with an eigenvector in $F_A=\Fix P_A$. Employing Lemma~\ref{lem:resolvent-under-domination}, we get that $i\beta$ is also a singularity of $R_B$. Consequently, \cite[Proposition~2.14(c)]{RabigerWolff2000} implies that $i\beta \in \spec(B)$.
\end{proof}

\section{Asymptotic behaviour}
    \label{sec:convergence}

In the study of operator semigroups, convergence properties are of significant interest. Similar to the long-term behaviour of positive semigroups, eventually positive semigroups also exhibit certain restrictive properties. These restrictions significantly aid in proving the convergence to equilibrium for many eventually positive semigroups, whether individually or uniformly eventually positive, as demonstrated in \cite{AroraGlueck2021a}. In this section, we introduce additional conditions beyond eventual positivity to derive more comprehensive convergence results. These supplementary conditions are closely related to the theory of eventual positivity. In the first subsection, we explore a weaker form of eventual \emph{strong} positivity (see, \cite{DanersGlueckKennedy2016b, DanersGlueck2018b, DanersGlueckKennedy2016a} or Section~\ref{sec:non-empty-spectrum}). In the subsequent subsection, we examine persistently irreducible semigroups.
 
\subsection*{Convergence of eventually positive semigroups}

For individually eventually positive semigroups on Banach lattices, it was shown in \cite[Corollary~2.2]{AroraGlueck2021a} that relative compactness of the semigroup orbits and boundedness of the intersection of point spectrum with the imaginary axis ensures the strong convergence of the semigroup operators. Along these lines, our main result of this subsection is as follows:

\begin{proposition}
    \label{prop:strong-convergence-sufficient}
    Let $E$ be a complex Banach lattice and let $(e^{tA})_{t \ge 0}$ be a  real and individually eventually positive semigroup on $E$ such that for every $0\lneq f\in E$, there exists $t_f\ge 0$ such that $e^{t_f A}f$ is a quasi-interior point of $E_+$.

    If $(e^{tA})_{t \ge 0}$ has relatively compact orbits, then exactly one of the following holds:
    \begin{enumerate}[\upshape (i)]
        \item The semigroup $(e^{tA})_{t \ge 0}$ converges strongly to $0$.
        \item The semigroup $(e^{tA})_{t \ge 0}$ converges strongly to the rank-one projection $u\otimes \varphi$ for a quasi-interior point $u$ of $E_+$ and a strictly positive functional $\varphi\in E'$ such that $u\in \ker A$ and $\varphi \in \ker A'$.
    \end{enumerate}
\end{proposition}

In Theorem~\ref{thm:strong-convergence-irreducible}, we give sufficient conditions which ensure that the strong positivity assumption in Proposition~\ref{prop:strong-convergence-sufficient} is fulfilled.
The statement of Proposition~\ref{prop:strong-convergence-sufficient} is inspired from the analogous result for positive (not necessarily strongly continuous) semigroups in \cite[Theorem~4.1]{GlueckWeber2020}. We recall that semigroups with relatively compact orbits are sometimes called \emph{almost periodic}. A complete characterization is given in \cite[Lemma~V.2.3]{EngelNagel2000}. In particular, bounded semigroups whose generators have compact resolvent have relatively compact orbits.

The proof of Proposition~\ref{prop:strong-convergence-sufficient} relies on Jacobs-de Leeuw-Glicksberg (JdLG) decomposition of eventually positive semigroups that was implicitly proved in \cite[Theorem~6.3.2]{Glueck2016}. The statement that we precisely need is given in the following lemma:

\begin{lemma}
    \label{lem:jdlg}
    On a complex Banach lattice $E$, 
    let $(e^{tA})_{t \ge 0}$ be an individually eventually positive $C_0$-semigroup.
    
    If $(e^{tA})_{t \ge 0}$ has relatively weakly compact orbits, then
    there exists a positive projection $P\in \calL(E)$ commuting with the semigroup $(e^{tA})_{t \ge 0}$ such that the following assertions hold.
    \begin{enumerate}[\upshape (a)]
        \item The semigroup $\left(e^{tA}\restrict{\Ima P}\right)_{t\geq 0}$ extends to a bounded positive $C_0$-group on $\Ima P$.
        \item The range $\Ima P$ is the closed linear span of eigenvectors of $A$ belonging to eigenvalues on the imaginary axis $i\bbR$.
        \item If $(e^{tA})_{t \ge 0}$ has relatively compact orbits, then the kernel of $P$ coincides with $\{ f\in E: \lim_{t\to \infty}e^{tA}f= 0\}$.
    \end{enumerate}
\end{lemma}

\begin{proof}
    The existence of a positive projection $P$ satisfying~(a) and~(b) is implicitly shown in the proof of \cite[Theorem~2.1]{AroraGlueck2021a}. For assertions~(c), we refer to \cite[Theorems~V.2.8 and~V.2.14]{EngelNagel2000}.
\end{proof}

For the proof of Proposition~\ref{prop:strong-convergence-sufficient}, we additionally need the following characterization of when a positive operator maps quasi-interior points to quasi-interior points. 
The lemma is taken from \cite[Proposition~2.21]{GlueckWeber2020} where it is stated for the more general case of ordered Banach spaces (the notion \emph{almost interior point} from this reference is, on Banach lattices, equivalent to quasi-interior point).
For the convenience of the reader, we include the simple proof.

\begin{lemma}
    \label{lem:qi-to-qi}
    Let $T: E \to F$ be a positive operator between Banach lattices and assume that $E$ contains a quasi-interior point of $E_+$. 
    The following are equivalent.
    \begin{enumerate}[\upshape (i)]
        \item
        There exists $x \in E_+$ such that $Tx$ is a quasi-interior point of $F_+$. 

        \item
        The adjoint operator $T': F' \to E'$ is strictly positive, meaning that $T'\varphi \gneq 0$ for every $0 \lneq \varphi \in F'$. 

        \item
        For every quasi-interior point $u$ of $E_+$, the image $Tu$ is a quasi-interior point of $F_+$.
    \end{enumerate}
\end{lemma}

\begin{proof}
    We recall again that a vector $u \in E$ is quasi-interior point of $E_+$ if and only if $\langle \psi, u \rangle > 0$ for all $0 \lneq \psi \in E'$ 
    \cite[Theorem~II.6.3]{Schaefer1974}. Of course, (iii) implies (i).
    
    ``(i) $\Rightarrow$ (ii)'':
    Let $x \in E_+$ be such that $Tx$ is a quasi-interior point of $F_+$ and let $0 \lneq \varphi \in F'$. 
    Then $\langle T' \varphi, x \rangle = \langle \varphi, Tx \rangle > 0$, 
    so $T' \varphi \not= 0$.
    
    ``(ii) $\Rightarrow$ (iii)'':
    Let $u \in E_+$ be a quasi-interior point 
    and let $0 \lneq \varphi \in F'$. 
    Due to strict positivity of $T'$, one has $T'\varphi \gneq 0$, 
    so $\langle \varphi, Tu \rangle = \langle T'\varphi, u \rangle > 0$. 
    This shows that $Tu$ is a quasi-interior point of $F_+$.
\end{proof}

\begin{proof}[Proof of Proposition~\ref{prop:strong-convergence-sufficient}]
    Let $P$ be the projection as in Lemma~\ref{lem:jdlg}. If $\ker P=E$, then by Lemma~\ref{lem:jdlg}(c), the semigroup $(e^{tA})_{t \ge 0}$ converges strongly to $0$.

    Next, suppose that $\ker P\ne E$. Then $\dim \Ima P\ge 1$ and $\spb(A)=0$ by boundedness of the semigroup and Lemma~\ref{lem:jdlg}(b).
    Since $P$ is a positive projection, $\Ima P$ is a Banach lattice with respect to an equivalent norm \cite[Proposition~III.11.5]{Schaefer1974}.
    We show that $\dim \Ima P=1$. To this end, it suffices to show that every positive non-zero element in $\Ima P$ is a quasi-interior point of $(\Ima P)_+$ (this sufficiency is proved in \cite[Lemma~5.1]{Lotz1968}, for an English translation, see \cite[Remark~5.9]{Glueck2018}).
    So, let $0\lneq f\in \Ima P$. 
    By assumption, there exists $t_f\ge 0$ such that $e^{t_fA}f$ is a quasi-interior point of $E_+$.
    Furthermore, since the semigroup is individually eventually positive, we can choose $t_f$ such that $e^{tA}f\ge 0$ for all $t\ge t_f$.
    Now, as $P$ is a positive projection, it maps quasi-interior points of $E_+$ to quasi-interior points of $(\Ima P)_+$; see, for instance, \cite[Proposition~9.3.1]{Arora2023}. Therefore, $e^{t_fA}f=e^{t_fA}Pf= Pe^{t_fA}f$ is a quasi-interior point of $(\Ima P)_+$. Thus, $e^{-t_f A}$ maps the element $e^{2t_fA}f \in (\Ima P)_+\subseteq E_+$ to the quasi-interior point $e^{t_fA}f$ of $(\Ima P)_+$. 
    Additionally, positivity of the operator $\left(e^{t_f A}\restrict{\Ima P}\right)^{-1}$ is known from Lemma~\ref{lem:jdlg}(a).
    Lemma~\ref{lem:qi-to-qi} thus implies that $e^{-t_f A}$ maps all quasi-interior point of $E_+$ to quasi-interior points of $(\Ima P)_+$. In particular, $f=e^{-t_f A}(e^{t_fA}f)$ is a quasi-interior of $(\Ima P)_+$, as intended.

    As $P$ is a positive projection with $\dim \Ima P=1$, there exists $0\lneq u \in E$ and $0\lneq \varphi \in E'$ such that $P=u\otimes \varphi$. Moreover, assertion~(b) in Lemma~\ref{lem:jdlg} and $\dim \Ima P=1$ even imply that $\pntSpec(A)\cap i\bbR$ is bounded (eigenvectors corresponding to distinct eigenvalues are linearly independent). Since the semigroup is individually eventually positive and $\spb(A)=0$, so $0\in \spec(A)$ by \cite[Theorem~7.6]{DanersGlueckKennedy2016a} and in turn, $\pntSpec(A)\cap i\bbR = \{0\}$ due to \cite[Theorem~2.1]{AroraGlueck2021a}. It follows that $e^{tA}$ acts as the identity operator on $\Ima P$ for each $t \ge 0$. Consequently, $u\in \ker A$, the semigroup $(e^{tA})_{t\ge 0}$ converges strongly to $P$ (cf. \cite[Corollary~2.2]{AroraGlueck2021a}),
    and -- by assumption -- $u=e^{t_uA}u$ is a quasi-interior point of $E_+$. Likewise, each $e^{tA'}$ acts as the identity operator on $\Ima P'$. So, $\varphi \in \ker A'$ and if $0\lneq f \in E$, then 
    \[
        \duality{\varphi}{f}=\duality{\varphi}{e^{t_f A} f} \ne 0;
    \]
    because $e^{t_f A} f$ is a quasi-interior point of $E_+$. In other words, $\varphi$ is strictly positive.
\end{proof}

\subsection*{Convergence of persistently irreducible semigroups}

In this subsection, we impose the additional assumption of persistent irreducibility on our semigroup. At first, we give a direct application of Proposition~\ref{prop:strong-convergence-sufficient} to obtain convergence in the strong operator topology. Thereafter, we shift our attention to operator norm convergence which has been characterized for uniformly eventually positive semigroups in \cite[Section~5]{AroraGlueck2021a}.

\begin{theorem}
    \label{thm:strong-convergence-irreducible}
    Let $E$ be a complex Banach lattice and let $(e^{tA})_{t \ge 0}$ be a real, analytic, uniformly eventually positive, and persistently irreducible semigroup on $E$.

    If $(e^{tA})_{t \ge 0}$ has relatively compact orbits, then exactly one of the following holds:
    \begin{enumerate}[\upshape (i)]
        \item The semigroup $(e^{tA})_{t \ge 0}$ converges strongly to $0$.
        \item The semigroup $(e^{tA})_{t \ge 0}$ converges strongly to the rank-one projection $u\otimes \varphi$ for a quasi-interior point $u$ of $E_+$ and a strictly positive functional $\varphi\in E'$ such that $u\in \ker A$ and $\varphi \in \ker A'$.
    \end{enumerate}
\end{theorem}

\begin{proof}
    Due to Proposition~\ref{prop:strong-convergence-sufficient}, we only need to show that the semigroup eventually maps every positive non-zero element to a quasi-interior point of $E_+$.
    But this is indeed the case for 
    analytic, uniformly eventually positive, and persistent irreducible semigroups by \cite[Proposition~3.12]{AroraGlueck2024}. 
\end{proof}

The spectral bound of a semigroup generator being a dominant spectral value is a necessary and one of the sufficient conditions for the convergence of the semigroup in the operator norm topology \cite[Theorem~2.7]{Thieme1998}. For this reason, it is often useful to have sufficient conditions that ensure that the spectral bound is dominant. For uniformly eventually positive semigroups that are norm continuous at infinity, certain conditions are collected in \cite[Theorem~7.3.3]{Arora2023}. In addition, if the semigroup is persistently irreducible, then \cite[Theorem~7.3.3(b)]{Arora2023} reads as follows:

\begin{proposition}
    \label{prop:dominant-spectral-value}
    Let $E$ be a complex Banach lattice and let $(e^{tA})_{t \ge 0}$ be a  real, uniformly eventually positive, and persistently irreducible semigroup on $E$.

    If $(e^{tA})_{t \ge 0}$ is norm continuous at infinity and $\spb(A)>-\infty$ is a pole of the resolvent $\Res(\argument, A)$, then it is a dominant spectral value of $A$.
\end{proposition}

A spectral value of a closed operator $A$ on a Banach space is called a \emph{Riesz point} of $A$ if it is a pole of the resolvent $\Res(\argument,A)$ and the corresponding spectral projection has finite-dimensional range. 

\begin{proof}[Proof of Proposition~\ref{prop:dominant-spectral-value}]
    Since $(e^{tA})_{t \ge 0}$ is persistently irreducible, individually eventually positive, and $\spb(A)$ is a pole of the resolvent $\Res(\argument,A)$, so by \cite[Corollary~4.4(c)]{AroraGlueck2024}, it is even a simple pole and a Riesz point of $A$. 

    Uniform eventual positivity and norm continuity at infinity now guarantee that $\spb(A)$ is a dominant spectral value of $A$ (see \cite[Lemma~5.3]{AroraGlueck2021a} or \cite[Theorem~7.3.3]{Arora2023}).
\end{proof}

In \cite[Theorem~3.5]{MartinezMazon1996}, it was shown that if $(e^{tA})_{t \ge 0}$ is a positive semigroup that is norm continuous at infinity such that $\spb(A)$ is a Riesz point of $A$, then $(e^{tA})_{t \ge 0}$ is even essentially compact. Using Proposition~\ref{prop:dominant-spectral-value}, we generalise this below.

Recall that for a densely defined closed linear operator $A$, the \emph{(Browder) essential spectrum} can be defined as
\[
    \essSpec(A):= \{ \lambda \in \spec(A): \lambda \text{ is not a Riesz point of }A\};
\]
cf. \cite[Theorem~1]{Lay1968},
and the \emph{essential spectral bound} is then defined as
\[  
    \essSpb(A):= \sup_{\lambda \in \essSpec(A)} \re \lambda.
\]
Of course, $\essSpb(A)\le \spb(A)$.
In addition, if $A$ generates a $C_0$-semigroup, then the semigroup is called \emph{essentially compact} if it is norm continuous at infinity and $\essSpb(A)<\spb(A)$; see \cite[Theorem~1.5]{MartinezMazon1996}.

\begin{proposition}
    Let $E$ be a complex Banach lattice and let $(e^{tA})_{t \ge 0}$ be a  real, uniformly eventually positive, and persistently irreducible semigroup on $E$.

    If $(e^{tA})_{t \ge 0}$ is norm continuous at infinity and $\spb(A)>-\infty$ is a pole of the resolvent $\Res(\argument,A)$, then $(e^{tA})_{t \ge 0}$ is essentially compact.
\end{proposition}

\begin{proof}
    Since $(e^{tA})_{t \ge 0}$ is norm continuous at infinity, we only need to show that $\essSpb(A)<\spb(A)$.
    Suppose not, then $\essSpb(A)=\spb(A)$ and so there exists a sequence $(\lambda_n)$ in $\essSpec(A)$ such that $\re \lambda_n \to \spb(A)$. Again employing norm continuity at infinity of the semigroup, we can -- due to \cite[Theorem~1.9]{MartinezMazon1996} -- find $\epsilon>0$ such that
    \[
        \{ \lambda \in \essSpec(A) : \re \lambda \ge \spb(A)-\epsilon\}
    \]
    is compact. As a result, there is a subsequence $(\lambda_{n_k})$ of $(\lambda_n)$ converging to some $\lambda \in \essSpec(A)$. In particular, $\re \lambda =\spb(A)$ and so $\lambda \in \perSpec(A)=\{\spb(A)\}$, where the equality holds due to Proposition~\ref{prop:dominant-spectral-value}. In turn, $\spb(A)=\lambda \in \essSpec(A)$.

    However, as $(e^{tA})_{t \ge 0}$ is persistently irreducible, individually eventually positive, and $\spb(A)$ is a pole of the resolvent $\Res(\argument,A)$, so by \cite[Corollary~4.4(c)]{AroraGlueck2024}, it is even a Riesz point of $A$ contradicting that $\spb(A) \in \essSpec(A)$.
\end{proof}

Sufficient conditions for uniform convergence of semigroups to a finite-rank operator were provided by Webb in \cite[Proposition~2.3]{Webb1987} and Thieme in \cite[Theorem~3.3]{Thieme1998}. The latter for the case of positive and eventually positive semigroups was improved in \cite[Theorem~3.4]{Thieme1998} and \cite[Theorem~5.2]{AroraGlueck2021a} respectively. With the additional assumption of persistent irreducibility, we can further improve these and give the exact description of the limit operator:

\begin{theorem}
    \label{thm:uniform-exponential-balancing}
    On a complex Banach lattice $E$, let $(e^{tA})_{t \ge 0}$ be a  real, uniformly eventually positive, and persistently irreducible semigroup.
    Suppose that $(e^{tA})_{t \ge 0}$ is norm continuous at infinity and $\spb(A)>-\infty$ is a pole of the resolvent $\Res(\argument, A)$.

    Then there exists a quasi-interior point $u$ of $E_+$ and a strictly positive functional $\varphi\in E'$ with $\duality{\varphi}{u}=1$ such that the operators $e^{t(A-\spb(A))}$ converge uniformly to the rank-one operator $u\otimes \varphi$ as $t\to \infty$.
\end{theorem}

\begin{proof}
    Since $\spb(A)$ is a pole of the resolvent $\Res(\argument,A)$, so due to uniform eventual positivity and persistent irreducibility of $(e^{tA})_{t \ge 0}$, the pole order is one and the corresponding spectral projection $P$ is given by $P=u\otimes \varphi$ for a quasi-interior point $u$ of $E_+$ and a strictly positive functional $\varphi \in E'$ \cite[Corollary~4.4(a)]{AroraGlueck2024}. 
    
    On the other hand, Proposition~\ref{prop:dominant-spectral-value} tells us that $\spb(A)$ is a dominant spectral value of $A$. In particular, all conditions in \cite[Proposition~8.2.1]{Arora2023} that guarantee the uniform convergence of the rescaled semigroup $(e^{t(A-\spb(A))})_{t\ge 0}$ to $P$ are fulfilled.
\end{proof}
    
    Note that if $(e^{tA})_{t \ge 0}$ is an essentially compact semigroup on a Banach space, then $\spb(A)\notin \essSpec(A)$ and so if $\spb(A)\in \spec(A)$, then it must, in particular, be a pole of the resolvent $\Res(\argument,A)$.
    This observation gives us the following  generalisation of \cite[Theorem~14.18]{BatkaiKramarRhandi2017} with the aid of Theorem~\ref{thm:uniform-exponential-balancing}.

\begin{corollary}
    \label{cor:uniform-exponential-balancing}
    Let $E$ be a complex Banach lattice and let $(e^{tA})_{t \ge 0}$ be a  real, uniformly eventually positive, and persistently irreducible semigroup on $E$.

    If $(e^{tA})_{t \ge 0}$ is essentially compact, then there exists a quasi-interior point $u$ of $E_+$ and a strictly positive functional $\varphi\in E'$ with $\duality{\varphi}{u}=1$ such that the operators $e^{t(A-\spb(A))}$ converge uniformly to the rank-one operator $u\otimes \varphi$ as $t\to \infty$.
\end{corollary}

\begin{proof}
    Essential compactness of the semigroup guarantees that $\spb(A)>-\infty$. Combined with the eventual positivity of the semigroup, this guarantees that $\spb(A)\in \spec(A)$ by \cite[Theorem~7.6]{DanersGlueckKennedy2016a}. Again using the fact that the semigroup is essentially compact, we obtain from the observation preceding the result that $\spb(A)$ is a pole of the resolvent. The assertion is now a consequence of Theorem~\ref{thm:uniform-exponential-balancing}.
\end{proof}

We close this section with an application of the above result to a coupled system; cf. \cite[Example~5.6]{AroraGlueck2024}.

\begin{example}
    \label{ex:uniform-exponential-balancing}
    On $E_1:= \bbC^3$, consider the semigroup $(e^{tA_1})_{t\ge 0}$ generated by
    \[
        A := \begin{bmatrix}
                7 & -1 & 3\\
                -1 & 7 & 3\\
                3 & 3 & 3
        \end{bmatrix}
    \]
    which is (uniformly) eventually positive according to \cite[Example~5.2]{AroraGlueck2024}. Moreover, since each entry of $A$ is non-zero, the directed graph associated to $A$ is strongly connected. Hence, $(e^{tA_1})_{t\ge 0}$ is even (persistently) irreducible by Proposition~\ref{prop:irreducibility-matrix-semigroup}.

    Next, for $p\in [1,\infty)$, let $A_2$ denote the Dirichlet Laplacian generating an analytic, positive, and (persistently) irreducible semigroup $(e^{tA_2})_{t\ge 0}$ on $E_2:= L^p(0,1)$. In addition, we define positive operators $B_{12}:E_2 \to E_1$ and $B_{21}:E_1\to E_2$ respectively by
    \[
        f\mapsto \int_0^1 f(x)\dx x\ e_3\quad\text{and}\quad z\mapsto z_3\one_{(0,1)};
    \]
    where $e_3 \in \bbC^3$ denotes the third canonical unit vector. Then \cite[Corollary~5.4]{AroraGlueck2024} tells us that the coupled semigroup $(e^{tC})_{t\ge 0}$ on $E_1\times E_2$ generated by
    \[
        C := 
        \begin{bmatrix}
            A_1 &     \\
                & A_2
        \end{bmatrix}
        +
        \begin{bmatrix}
                    & B_{12}    \\
            B_{21}  & 
        \end{bmatrix},
    \]
    is uniformly eventually positive and persistently irreducible. Of course, $(e^{tC})_{t\ge 0}$ is also analytic and $C$ has compact resolvent with non-empty spectrum. Therefore, $(e^{tC})_{t\ge 0}$ is essentially compact due to \cite[Corollary~1.7]{MartinezMazon1996}. Consequently, Corollary~\ref{cor:uniform-exponential-balancing} guarantees that the operators $e^{t(C-\spb(C))}$  converge uniformly to a rank-one operator as $t\to \infty$.
\end{example}

\begin{remark}
    Since boundedness of the semigroup generated by $C-\spb(C)$ is not immediately clear in Example~\ref{ex:uniform-exponential-balancing}, one could not apply the previously known criterion \cite[Theorem~5.1]{AroraGlueck2021a} for uniform convergence of eventually positive semigroups. This roadblock could be avoided due to the presence of persistent irreducibility.
\end{remark}

\section{Non-empty spectrum}
    \label{sec:non-empty-spectrum}

Positive (and in particular, eventually positive) semigroups can have void spectrum \cite[Example~B-III-1.2]{Nagel1986}. In fact, even positive irreducible semigroups with empty spectrum can be constructed \cite[Example~C-III-3.6]{Nagel1986}. Sufficient conditions that guarantee that a positive irreducible semigroup has a non-empty spectrum are given in \cite[Theorem~C-III-3.7]{Nagel1986}. For persistently irreducible eventually positive semigroups, the non-emptiness of the spectrum was investigated in \cite[Corollary~4.6 and Theorem~4.7]{AroraGlueck2024} on the spaces of continuous functions. In this section, we give yet another sufficient condition which seems to be new even for the case of positive irreducible semigroups.

\begin{theorem}
    \label{thm:non-empty-spectrum}
    Let $E$ be a complex Banach lattice and let $(e^{tA})_{t \ge 0}$ be a  real, uniformly eventually positive, and persistently irreducible semigroup on $E$.

    Suppose that that there exists $u\in E_+$ and $t_0\ge 0$ such that $e^{t_0 A} E$ maps into $E_u$ and contains a quasi-interior point of $(E_u)_+$, then $\spec(A)$ is non-empty.
\end{theorem}

For positive irreducible semigroups on reflexive Banach lattices, the existence of $u\in E_+$ and $t_0\ge 0$ such that $e^{t_0A}\subseteq E_u$ is sufficient for $A$ to have non-void spectrum \cite[Theorem~C-III-3.7(e)]{Nagel1986}. In contrast, we make no reflexivity assumptions in Theorem~\ref{thm:non-empty-spectrum}.  

\begin{remark}
    The condition $e^{t_0A}E\subseteq E_u$ for some $t_0\ge 0$ was known as the \emph{smoothing assumption} in \cite{DanersGlueckKennedy2016b, DanersGlueck2018b} and played a crucial role in characterizing eventual strong positivity of semigroups (see also, \cite{DanersGlueck2017}). Perturbation of this condition and its application to eventual positivity has been recently investigated in \cite{AroraMui2025}.
\end{remark}

If $E$ is a complex Banach lattice in which every non-empty relatively compact set is order bounded, then every individually eventually positive semigroup $(e^{tA})_{t \ge 0}$ on $E$ satisfies $\spb(A)=\gbd(A)$; see \cite[Theorem~6.2.1]{Arora2023} for a proof. Actually, the proof in the aforementioned reference only requires that for each $f\in E_+$, there exists $t_f\ge 0$ and $g\in E$ such that $0 \le e^{tA}f \le g$ on $[t_f,t_f+1]$. We state this observation separately in the below lemma as it plays a crucial role in the proof of Theorem~\ref{thm:non-empty-spectrum}.

\begin{lemma}
    \label{lem:lyapunov}
    Let $(e^{tA})_{t \ge 0}$ be an individually eventually positive semigroup on a Banach lattice $E$ such that for each $f\in E_+$, there exists $t_f\ge 0$ and $g\in E$ with
    \[
        0 \le e^{tA}f \le g \quad \text{for all}\quad t\in [t_f,t_f+1].
    \]
    Then $\spb(A)=\gbd(A)$.
\end{lemma}

Note that for elements $x,y$ of a Banach lattice $E$, we use the usual notation $x\succeq y$ to mean that there exists a constant $c>0$ such that $x\ge cy$.

\begin{proof}[Proof of Theorem~\ref{thm:non-empty-spectrum}]
    First of all, note that since $e^{t_0 A}E\subseteq E_u$, so by the semigroup law $e^{tA}E\subseteq E_u$ for all $t\ge t_0$.
    Moreover, due to uniform eventual positivity, we can choose $\tau\ge t_0$ such that $e^{tA}\ge 0$ for all $t\ge \tau$. In particular, for each $f \in E_+$, compactness of $[\tau,\tau+1]$ guarantees the existence of $c_f>0$ such that $0\le e^{tA}f\le c_f u$ for all $t\in [\tau,\tau+1]$. Therefore, the assumptions of Lemma~\ref{lem:lyapunov} are fulfilled and so we only need to show that $\gbd(A)>-\infty$.

    Next, letting $v:= e^{\tau A}u \in E_+$, we see that $e^{tA}E \subseteq E_v$ for all $t\ge t_0+\tau$ and so by the closed graph theorem $e^{tA}\in \calL(E, E_v)$ for all $t\ge t_0+\tau$. Observe that persistent irreducibility of the semigroup guarantees that it is not nilpotent, which is why $v$ cannot be non-zero.
    
    Now, by assumption, there exists $h \in E$ such that $g:=e^{t_0 A}h\in E_u$ is a quasi-interior point of $(E_u)_+$. This means that $g\succeq u$ and in turn, $e^{\tau A}g\succeq  e^{\tau A}u=v$. In particular, $e^{\tau A}g \in E_v$ is a quasi-interior point of $(E_v)_+$ as well. In other words, $\duality{\nu}{e^{\tau A}g} >0$ for every $0\lneq \nu \in (E_v)'$.

    Let $\Delta$ denote the set of all positive linear functionals on $E_v$ with norm equal to $1$. By Banach-Alaoglu theorem, $\Delta$ is a weak${}^*$-compact subset of $(E_v)'$. Fix $\nu \in \Delta$ and let $(e^{(t_0+\tau) A})'$ denote the dual of $e^{(t_0+\tau) A}\in \calL(E, E_v)$. Then
    \[
        \duality{(e^{(t_0+\tau)A})'\nu}{h} = \duality{\nu}{e^{\tau A}g} >0
    \]
    because $0\lneq \nu \in (E_v)'$. As a result, $0\lneq (e^{(t_0+\tau) A})'\nu \in E'$; using the positivity of $e^{(t_0+\tau) A}$. So, by persistent irreducibility of the semigroup, we can find a time $t_{\nu}\ge t_0+\tau$ such that $\duality{\nu}{e^{t_{\nu}A}v} >0$; see \cite[Theorem~3.6]{AroraGlueck2024}. As $\nu \in \Delta$ was arbitrary, weak${}^*$-compactness of $\Delta$ yields $t_1,t_2\ldots,t_n \ge t_0+\tau$ such that
    \[
        \sum_{k=1}^n \duality{\nu}{e^{t_k A} v} > 0 \quad \text{for all}\quad \nu \in \Delta.
    \]
    Setting $T:=\sum_{k=1}^n  e^{t_k A}$, we obtain that $Tv \in E_v$ is a quasi-interior point of $(E_v)_+$, again from the characterization of quasi-interior points in \cite[Theorem~II.6.3]{Schaefer1974}. This means that there exists $\delta>0$ such that $Tv \ge \delta v$ and so positivity due to the positivity of $T$, we even have $T^nv \ge \delta^n v$ for all $n\in \bbN$. As $v\ne 0$, we infer from the Gelfand formula  that the spectral radius of $T$ satisfies $\spr(T)\ge \delta>0$. By the subadditivity of the spectral radius on commuting operators, there exists $k\in \{1,\ldots,n\}$ such that
    $
        e^{t_k \gbd(A) }=\spr(e^{t_k A})>0
    $
    which implies that $\gbd(A)>-\infty$.
\end{proof}

\subsection*{Eventual strong positivity}

A $C_0$-semigroup $(e^{tA})_{t \ge 0}$ on a Banach lattice $E$ is called \emph{uniformly eventually strongly positive} if there exists $t_0\ge 0$ such that $e^{tA}f$ is a quasi-interior point of $E_+$ for all $0\lneq f\in E$ and all $t\ge t_0$. Furthermore, for $u\in E_+$ if there exists $t_0\ge 0$ such that $e^{tA}f\succeq u$ for all $0\lneq f\in E$ and all $t\ge t_0$, then $(e^{tA})_{t \ge 0}$ is called \emph{uniformly eventually strongly positive with respect to $u$}. Sufficient conditions for a semigroup to belong to the latter class are given in \cite{DanersGlueck2018b}. In fact, all available concrete examples of uniformly eventually positive semigroups fall into this category. In particular, they are all persistently irreducible due to \cite[Proposition~3.3]{AroraGlueck2024}. Here, we give a condition that ensures when a  uniformly eventually strongly positive semigroup is even uniformly eventually strongly positive with respect to $u$.

\begin{proposition}
    \label{prop:strong-positivity-with-respect-to-u}
    Let $E$ be a complex Banach lattice and let $(e^{tA})_{t \ge 0}$ be a $C_0$-semigroup on $E$
    such that there exists $u\in E_+$ and  $t_0\ge 0$ such that $e^{t_0 A} E_+$ maps into $(E_u)_+$ and contains a quasi-interior point of $(E_u)_+$.

    If $(e^{tA})_{t \ge 0}$ is uniformly eventually strongly positive, then it is even uniformly eventually strongly positive with respect to $u$.
\end{proposition}

\begin{proof}
    Observe that by assumption $e^{t_0A}$ is a positive operator that maps a positive element of $E$ to a quasi-interior point of $(E_u)_+$. In addition, using the fact that the cone of $E$ is generating and the closed graph theorem, we get $e^{t_0 A}\in \calL(E, E_u)$. Therefore, we may infer from Lemma~\ref{lem:qi-to-qi} that $e^{t_0A}$ maps quasi-interior points of $E_+$ to quasi-interior points of $(E_u)_+$. 

    On the other hand, by uniform eventual strong positivity of the semigroup,
    we can choose $\tau\ge 0$ such that $e^{tA}f$ is a quasi-interior point of $E_+$ for every $0\lneq f\in E$ and all $t\ge \tau$. As a consequence, $e^{(t+t_0)A}f$ is a quasi-interior point of $(E_u)_+$ for every $0\lneq f\in E$ and all $t\ge \tau$. Since every quasi-interior point $x$ of $(E_u)_+$ satisfies $x\succeq u$, the assertion follows.
\end{proof}

Since uniformly eventually positive, analytic, and (persistently) irreducible semigroups are even uniformly eventually strongly positive \cite[Proposition~3.12]{AroraGlueck2024}, the following immediately follows from Proposition~\ref{prop:strong-positivity-with-respect-to-u}.

\begin{corollary}
    Let $E$ be a complex Banach lattice and let $(e^{tA})_{t \ge 0}$ be a $C_0$-semigroup on $E$
    such that there exists $u\in E_+$ and $t_0\ge 0$ such that $e^{t_0 A} E_+$ maps into $(E_u)_+$ and contains a quasi-interior point of $(E_u)_+$.

    If $(e^{tA})_{t \ge 0}$ is uniformly eventually positive, analytic, and (persistently) irreducible then it is even uniformly eventually strongly positive with respect to $u$.
\end{corollary}

\subsection*{Acknowledgements} 

I am deeply grateful to Jochen Glück for various discussions and for sharing several ideas. His insightful comments and suggestions significantly enhanced this paper, shaping it into its present form. I also extend my thanks to the anonymous referee of \cite{AroraGlueck2024} for suggesting the investigation of Theorem~\ref{thm:cyclicity-peripheral} and to the anonymous referee of this article for their constructive feedback that helped improve some oversights in the earlier version.

The article was initiated while the author was preparing for a plenary talk delivered at the MAT-DYN-NET 24 that took place at the University of Minho, Portugal in February 2024. The author acknowledges COST Action 18232 for the financial support to attend this conference.

\bibliographystyle{plainurl}
\bibliography{literature}

\end{document}